\documentclass[twoside,12pt]{amsart}
\usepackage{amsmath,latexsym,amssymb,mathptm, times,verbatim, enumerate,mathcomp,lscape}
    \usepackage{longtable}\usepackage{tikz,stmaryrd}
\usepackage{pdfpages}
\usepackage{calligra}
\usepackage{calrsfs}
\usepackage[mathscr]{euscript}
\input amssym.def
\input amssym 
\input xypic
\input xy
\xyoption{all}
\usepackage[right=1.5in,left=1.5in,top=1in,bottom=1in]{geometry}

\def\id{\operatorname{id}}
\def\g{\gamma}

\def\ep{\epsilon}
\def\e{\epsilon}

\def\DGGamma{\operatorname{DG}\Gamma}
\def\DG{\operatorname{DG}}
\def\comult{\operatorname{comult}}
\def\proj{\operatorname{proj}}
\def\NNY{\lambda}

\def \ms{\medskip}
\def \bs{\bigskip}
\def\bw{\bigwedge}

\def\mfp{\mathfrak p}

\def\la{\langle}
\def\ra{\rangle}

\def\rank{\operatorname{rank}}

\def\p{\oplus}
\def\im{\operatorname{im}}

\def\m{\mathfrak m}

\def\a{\alpha}

\def\Hom{\operatorname{Hom}}
\def\kk {\pmb k}
\def\codim{\operatorname{codim}}
\def\Tor{\operatorname{Tor}}
\def\t{\otimes}
\def\w{\wedge}
\def\HH{\operatorname{H}}

\def\ts{\textstyle}
\def\Ass{\operatorname{Ass}}
\def\mfq{\mathfrak q}

\newtheorem{theorem}{Theorem}[section]
\newtheorem{lemma}[theorem]{Lemma}
\newtheorem{corollary}[theorem]{Corollary}
\newtheorem{proposition}[theorem]{Proposition}
\newtheorem{proposition-no-advance}[equation]{Proposition}

\newtheorem{claim-no-advance}[equation]{Claim}
\newtheorem{observation}[theorem]{Observation}

\newtheorem{quick consequences}[theorem]{Quick Consequences}

\theoremstyle{definition}

\newtheorem{definition-no-advance}[equation]{Definition}

\newtheorem{facts and definitions}[theorem]{Facts and Definitions}
\newtheorem{definition}[theorem]{Definition}
\newtheorem{remark}[theorem]{Remark}
\newtheorem{remark-no-advance}[equation]{Remark}
\newtheorem{remarks-no-advance}[equation]{Remarks}

\newtheorem{data}[theorem]{Data}

\newtheorem{conventions}[theorem]{Conventions}

\newtheorem{careful calculation}[theorem]{Careful Calculation}

\newtheorem{present summary}[theorem]{Present Summary}

\newtheorem{example}[theorem]{Example}

\newtheorem{further reductions}[theorem]{Further Reductions}

\newtheorem{chunk}[theorem]{}
\newtheorem{chunk-no-advance}[equation]{}

\newtheorem{marching orders}[theorem]{Marching Orders}

\newtheorem{circle the wagons}[theorem]{Circle the wagons}

\newtheorem*{Remark}{Remark}

\numberwithin{equation}{theorem}
\numberwithin{table}{theorem}

\begin{document}

\baselineskip=16pt

\title[Resolutions of length four which are Differential Graded Algebras]
{Resolutions of length four which are Differential Graded Algebras}

\date{\today}

\author[A.~R.~Kustin]{Andrew R.~Kustin}
\address{Andrew R.~Kustin\\ Department of Mathematics\\ University of South Carolina\\\newline 
Columbia\\ SC 29208\\ U.S.A.} \email{kustin@math.sc.edu}

\subjclass[2010]{13D02, 16E45}

\keywords{Differential Graded Algebras, Poincar\'e duality}

\thanks{}

\begin{abstract}   Let $P$ be a commutative Noetherian ring and $F$ be a self-dual acyclic complex of finitely generated free $P$-modules. Assume that $F$ has length four and $F_0$ has rank one. We prove  that $F$ can be given the structure of a Differential Graded Algebra with Divided Powers; furthermore, the multiplication on $F$ exhibits Poincar\'e duality. This result is already known if $P$ is a local Gorenstein ring and $F$ is a minimal resolution. The purpose of the present paper is to remove the unnecessary hypotheses that $P$ is local, $P$ is Gorenstein, and $F$ is minimal.
\end{abstract}

\maketitle

\section{Introduction.}

Let $P$ be a commutative Noetherian ring and $F$ be a self-dual acyclic complex of finitely generated free $P$-modules. Assume that $F$ has length four and $F_0$ has rank one. In Theorem~\ref{main} we prove  that $F$ can be given the structure of a Differential Graded Algebra with Divided Powers; furthermore, the multiplication on $F$ exhibits Poincar\'e duality. This result is already known \cite{KM80,K87} if $P$ is a local Gorenstein ring and $F$ is a minimal resolution. The purpose of the present paper is to remove the unnecessary hypotheses that $P$ is local, $P$ is Gorenstein, and $F$ is minimal.

Tate \cite{T57} introduced  Differential Graded ($\DG$) Algebras into commutative algebra. He 
exhibited a $\DG$-resolution  
of the residue class field $\kk=P/\m$ 
 when $P$ is a local Noetherian ring  
with maximal ideal $\m$. Gulliksen \cite{G68} later proved that Tate's resolution is a minimal resolution. Buchsbaum and Eisenbud \cite{BE77} proved that every finite free resolution $F$ of length at most three and $\rank F_0=1$ is a $\DG$-algebra and they used this fact in their classification of grade three Gorenstein ideals.  

The study of $\DG$-algebra resolutions  was significantly motivated by  
the work of Avramov.   
Let $A$ be the ring $P/I$, where $I$ is an ideal in a regular local ring $(P,\m,\kk)$, and let $F$ be the minimal resolution of $A$ by free $P$-modules. 
If $F$ is a DG-algebra, then  Avramov \cite[Cor.~3.3]{Av78}  proved that the Eilenberg-Moore spectral sequence $E^2_{p,q}=\Tor^{\HH(K^A)}_{p,q}(\kk,\kk)\implies \Tor^{K^A}_{p+q}(\kk,\kk)$ degenerates, where $K^A$ is the Koszul complex associated to a minimal generating set of the maximal ideal of $A$. When this happens many questions about the ring $A$ may be translated into questions about the Koszul homology algebra $T=\HH(K^A)=\Tor^P(A,\kk)$. In particular, the Poincar\'e series of $A$ may be expressed in terms of the Poincar\'e series of $T$. The algebra $T$, although graded-commutative instead of commutative, is in many ways simpler than the original algebra $A$. For example, $T$ is always a finite-dimensional vector space. This philosophy has led to some striking theorems in the case when $A$ has small codimension or small linking number: the Poincar\'e series of finitely generated $A$-modules have been calculated  \cite{AKM}, the asymptotics of the Betti numbers   of 
 finitely generated $A$-modules have been determined \cite{Av89}, and the Bass series of finitely generated $A$-modules have been found \cite{Av12}. Consequences of this technique continue to be found.  The paper \cite{AISS15} uses $\DG$-algebra techniques to guarantee that 
if $A$ has small codimension or small linking number
and 
$M$ and $N$ are finitely generated $A$-modules with $\Tor^A_i(M,N)=0$ for all large $i$, then $M$ or $N$ has finite projective dimension.  The techniques of $\DG$-algebras 
 are used in \cite{AISSII} to 
prove that local rings of small codimension or small linking number which are not Gorenstein and are not embedded deformations are $G$-regular in the sense that every totally reflexive module over such a ring is free.

The previously mentioned applications of $\DG$-techniques apply especially to local rings and minimal resolutions. On the other hand, $\DG$-techniques continue to be interesting when the ring is not local or the resolution is not minimal. Consider two ideals $J\subseteq I$ in the ring $P$, where $J$ is generated by a regular sequence and a convenient resolution $F$ of $P/I$ is a $\DG$-algebra. In order to resolve the linked ideal $J:I$, one studies the mapping cone of the map of complexes from the Koszul complex which  resolves  
$P/J$ to $F$. One has great control of this map of complexes when one makes it a map of $\DG$-algebras. This observation is used in \cite{Ku19} to create matrix factorizations.

The Taylor resolution of a  
ring
 defined by a monomial ideal is a $\DG$-algebra (see, for example, \cite[2.1]{Av81}), but is usually not minimal. 
 Up-to-date information about $\DG$-structures on resolutions of rings defined by monomial ideals can be found in \cite{Ka19}. 
In particular, Katth\"an has proven that
there exists a monomial ideal whose hull resolution \cite{BS98} does not admit a $\DG$-structure;
there exists a monomial ideal whose Lyubeznik resolution \cite{L88} does not admit a $\DG$-structure;
and there exists a  
 monomial ideal whose minimal free resolution is supported on a simplicial complex, 
 in the sense of \cite[Construction 2.1]{BPS98}, nonetheless the minimal free resolution 
does not admit a $\DG$-structure.
On the other hand, Katth\"an has proven \cite[Thm.~2.1]{Ka19} that
in many situations  
if $F$ is a
 resolution 
of the cyclic module $P/J$,
 then it is enough to modify the first map of $F$ in order to ensure that the modified resolution, $F'$, admits the structure of a $\DG$-algebra. (In particular, $F'$ is a resolution of $P/(f)J$ for some regular element $f$ of $P$.)

   Examples of minimal resolutions which do not support a $\DG$-structure have been found by V. A. Khinich \cite[Appendix]{Av74},
 Avramov \cite[Ex.~2.2]{Av81},
 and Srinivasan \cite{Sr92} and \cite{Sr96}.

Roughly speaking, there are three ways to put a $\DG$-structure on a $P$-free resolution $F$ of a cyclic $P$-module $A=P/I$. The first approach is to observe that $F$ always has a multiplication which satisfies all of the $\DG$ axioms, except it is associative only up to homotopy. If sufficient additional hypotheses are imposed, then every choice of homotopy-associative multiplication is, in fact, associative. This approach works in \cite{BE77} when $\codim (A)\le3$,
and in \cite{Sr94} when $F$ is a graded resolution whose grading satisfies the inequality $$s_{a_j}+s_{b_k}+s_{c_\ell}<s_{(a+b+c+1)_i},$$ for all $a$, $b$, $c$, $i$, $j$, $k$, and $\ell$, where $F_a=\bigoplus_jP(-s_{a_j})$.
 
The second approach is to prove that if $F$ is sufficiently short, then a homotopy-associative multiplication can be modified in order to become associative ``on the nose''. This is the approach in \cite{KM80,K87}
 for Gorenstein local rings A of codimension four; 
and in \cite{Pa90p,Ku94-aci} 
 for local codimension four almost complete intersection rings $A$. 

 The third approach is to record an explicit multiplication table for $F$ and show that it satisfies all of the relevant axioms. This approach works if $A$ is a complete intersection (in this case, the resolution $F$ is an exterior algebra); if $A$ is defined by a monomial ideal (in this case $F$ is the Taylor resolution); if $A$ is one link from a complete intersection \cite{AKM}; if $A$ is a  
 Gorenstein ring two links from a complete intersection
\cite{KM83}; if $A$ is a   ring defined by the maximal minors of a matrix in equicharacteristic zero \cite{Sr89} (see also \cite{Ma}); if $A$ is 
a Gorenstein ring defined by a Huneke-Ulrich deviation two ideal \cite{Sr91} (see also \cite[Thm.~2.4]{Ku94}); or if $A$ is a Huneke-Ulrich almost complete intersection 
\cite{Ku95}.
 
\ms 
Our proof of the main result uses the second approach.
An outline of the proof follows. In this discussion, $F$ is the resolution
$$0\to F_4\to F_3\to F_2\to F_1\to F_0=P$$and $x_i$ and $y_i$ are elements of $F_i$.

In 
Lemma~\ref{Mar-12}
 we identify maps $\psi_3: F_1\t F_3\to F_4$ and $\psi_4: D_2F_2 \to F_4$ such that 
\begin{align}
&\text{$\psi_3$ and $\psi_4$ 
satisfy
 the product rule for $0=x_1\cdot x_4$ and $0=x_2\cdot x_3$}\label{1.a}\\ \intertext{and}
&\text{the induced maps $F_i\to\Hom_P(F_{4-i},F_4)$ are isomorphisms,}\label{1.b}
\end{align}
for $1\le i\le 3$.
It is not difficult to obtain (\ref{1.a}). If $P$ is a local ring, 
 $F$ is a minimal resolution, and (\ref{1.a}) holds,  then (\ref{1.b}) holds automatically; however, in the general situation, maps which satisfy (\ref{1.a}) must be modified in order to ensure that they also satisfy (\ref{1.b}).

Once (\ref{1.a}) and (\ref{1.b}) are obtained, we make no further modification of the maps $\psi_3$ and $\psi_4$. We place all of the responsibility for creating a multiplication on our ability to make an appropriate choice for
$\Psi_1: F_1\t F_1\to F_2$.  That is, we define $\Psi_2:F_1\t F_2\to F_3$ in terms of $\Psi_1$, $\psi_3$, and $\psi_4$ by using the requirement that the multiplication must associate: $y_1(x_1x_2)=(y_1x_1)x_2$. Hence, we define $\Psi_2$ to be the homomorphism which satisfies 
$$\psi_3(y_1\t \Psi_2(x_1\t x_2))=\psi_4(\Psi_1(y_1\t x_1)\cdot x_2),$$
for $x_i$ and $y_i$ in $F_i$. 
This definition makes sense because of (\ref{1.b}). 
 
The map $\Psi_1$ must satisfy 3 hypotheses:
\begin{align}
 &\text{$\Psi_1$ must satisfy one differential condition 
pertaining to 
$F_1\t F_1 \to F_2$,}\label{2.a}\\
&\text{$\Psi_1$ must satisfy one differential condition pertaining to $F_1\t F_2 \to F_3$,}\label{2.b} \intertext{and}
&\text{$\Psi_1$ must factor through $\ts\bw^2F_1$.}\label{2.c}
\end{align}
The precise statement of (\ref{2.a}) -- (\ref{2.c}) is given
 as Lemma~\ref{mult-table}; the proof that once these conditions are satisfied, then $\Psi_1$, $\Psi_2$, $\psi_3$, and $\psi_4$ give $F$ the structure of a $\DGGamma$-algebra with Poincar\'e duality 
is given in Section~\ref{prove-mult-table}.

The rest of the paper is devoted to obtaining a map $\Psi_1$ which satisfies (\ref{2.a}), (\ref{2.b}), and (\ref{2.c}). It is not difficult to find a map $\psi_1$ which satisfies (\ref{2.a}) and (\ref{2.b}). We make successive modifications of the original $\psi_1$ until we obtain a map $\Psi_1$ which satisfies all three conditions. The modifications in \cite{KM80} involved division by $2$; and the modifications in \cite{K87} involved division by $3$. In a local ring either two or three is a unit; and therefore, at least one of these two approaches is appropriate in any local ring. However, the ring $P$ in the present paper is not necessarily local. Instead of dividing by $2$ or $3$, we multiply by $2$ or $3$. In Section~\ref{prove-3} we obtain a map $\Psi_{1,3}$ which satisfies (\ref{2.c}) and which would satisfy (\ref{2.a}) and (\ref{2.b}) except the answer is $3$ times the desired answer. In Section~\ref{Inductive step.}, we obtain a map $\Psi_{1,2}$ which satisfies (\ref{2.c}) and which would satisfy (\ref{2.a}) and (\ref{2.b}) except the answer is $2^n$ times the desired answer, for some non-negative  integer $n$. To complete the proof of Theorem~\ref{main} we take $\Psi_1$ to be  an 
 appropriate integral linear combination of $\Psi_{1,2}$ and $\Psi_{1,3}$. 

The modifications which produce $\Psi_{1,3}$ and $\Psi_{1,2}$ are fundamentally different.  Lemma~\ref{Mar-12} produces maps $\psi_3$ and $\psi_4$ which satisfy (\ref{1.a}) and (\ref{1.b}). Lemma~\ref{Mar-12} also produces maps $\psi_1^\dagger$ and $\psi_2^\dagger$ so that $\psi_1^\dagger$, $\psi_2^\dagger$, $\psi_3$, and $\psi_4$ give $F$ the structure of an algebra which would be a $\DG$-algebra except the multiplication is homotopy-associative instead of being associative. The map $\Psi_{1,3}$ is obtained  by making one modification to the map $\psi_1^\dagger$ from Lemma~\ref{Mar-12}; this modification makes significant use of the map  $\psi_2^\dagger$ from Lemma~\ref{Mar-12}.

On the other hand, the construction of $\Psi_{1,2}$ completely ignores the maps $\psi_1^\dagger$ and $\psi_2^\dagger$ of Lemma~\ref{Mar-12}. Instead, a sequence of maps $\psi_1^{\la i\ra}:F_1\t F_1\to F_2$ is built. Each $\psi_1^{\la i\ra}$ 
would satisfy (\ref{2.a}) and (\ref{2.b}) except the answer is $2^{n_i}$ times the desired answer, for some non-negative  integer $n_i$; furthermore, for each $i$, there is a free summand $G_i$ of $F_1$ of rank $i$ such that the restriction of $\psi_1^{\la i\ra}$ to $G_i\t G_i$ factors through $\bw^2 G_i$. The initial map $\psi_1^{\la 0\ra}$ is obtained by way of a diagram chase in the double complex $\Hom_P(F_1\t F\t F,F_4)$, see Observation~\ref{Y22.5}. The modification of $\psi_1^{\la i-1\ra}$ into $\psi^{\la i\ra}$ takes place in Proposition~\ref{prop29.5}. The calculations which are used to prove Proposition~\ref{prop29.5} are contained in Section~\ref{Properties of the data.}. 

\tableofcontents

\section{Notation, conventions, and preliminary results.}\label{Prelims}

\begin{conventions}
\label{2.1} Let $P$ be a commutative Noetherian ring,
  $X$ be a free $P$-module and $Y$ be a $P$-module. 
\begin{enumerate}[\rm(a)]
\item \label{2.1.a}The
rules for a divided power algebra $D_\bullet X$ are recorded in \cite[Def.~1.7.1]{GL}  or  \cite[Appendix 2]{Ei95}.
(In practice these rules say that $x^{
(a)}$ behaves like $x^
a/(a!)$ would behave if $a!$ were
a unit in $P$.)

If $x$ and $x'$ are elements of $X$, then $x\cdot x'=x'\cdot x$ in $D_2(X)$. The co-multiplication homomorphism $$\comult:D_2X\to X\t_PX$$ sends $x^{(2)}$ to $x\t x$ and sends $x\cdot x'$ to $x\t x'+x'\t x$, for $x,x'$ in $X$. Often we will describe a homomorphism $\phi: D_2X\to Y$ by giving the value of $\phi(x^{(2)})$ for each $x\in X$. One then automatically knows the value of $\phi(x\cdot x')$, for $x,x'\in X$ because $$(x+x')^{(2)}=x^{(2)}+x\cdot x'+{x'}^{(2)}.$$ 
\item Let $T_nX$ represent $\underbrace{X\t\cdots \t X}_{\text{$n$ times}}$.
\item\label{2.1.c} A homomorphism $\phi:T_2X\to Y$ is called {\it alternating} if the composition
$$D_2X\xrightarrow{\comult}T_2X\xrightarrow{\phi}Y$$ is identically zero. In particular, $\phi$ is an alternating map if and only if $\phi$ factors through the natural quotient map $T_2X\to \bw^2X$.
\item If $I$ is a proper ideal of $P$, then the {\em grade} of $I$ is the length of the longest regular sequence in $I$ on $P$.
\item A complex $\cdots \to F_2\to F_1\to F_0\to 0$ is called {\it acyclic} if the only non-zero homology occurs in position zero.
\end{enumerate}
\end{conventions}

\begin{definition}
If $P$ is a ring and $A$, $B$, and $C$ are $P$-modules, then the $P$-module homomorphism $\phi: A\otimes_P B\to C$ is a {\it perfect pairing} if the induced $P$-module homomorphisms $A\to \mathrm{Hom}_P(B,C)$ and $B\to \mathrm{Hom}_P(A,C)$,  given by $a\mapsto \phi(a\otimes \underline{\phantom{X}})$ and $b\mapsto \phi(\underline{\phantom{X}}\otimes b)$, respectively, are isomorphisms.
\end{definition}

\begin{definition}
A {\it Differential Graded  
algebra} 
$F$ (written \underline{$\DG$-algebra}) over the commutative Noetherian ring $P$ is a complex of finitely generated free $P$-modules $(F, d)$:
$$\cdots 
\xrightarrow{d_2}F_1\xrightarrow{d_1}F_0=P,$$
together with a unitary, associative  multiplication $F\t_PF\to F$, which satisfies
\begin{enumerate}[\rm(a)]
\item $F_iF_j\subseteq F_{i+j}$,
\item $d_{i+j}(x_ix_j)=d_i(x_i)x_j+(-1)^i x_id_j(x_j)$,
\item $x_ix_j=(-1)^{ij}x_jx_i$, and
\item $x_i^2=0$, when  $i$ is odd,
\end{enumerate}
for  $x_{\ell}\in F_{\ell}$.
The $\DG$-algebra $F$ is called a \underline{$\DGGamma$-algebra} {\rm(}or a $\DG$-algebra with divided powers{\rm)} if, for each positive even index $i$ and each element $x_i$ of $F_i$, there is a family of elements $\{x_i^{(k)}\}$ which satisfy the divided power axioms of \ref{2.1}.(\ref{2.1.a}), and which also satisfy
$$d_{ik}(x_i^{(k)})=d_i(x_i) x_i^{(k-1)}.$$
The $\DG$-algebra $F$ 
 exhibits {\it Poincar\'e duality} if there there is an integer $m$ such that $F_i=0$ for $m< i$, $F_m$ is isomorphic to $P$, and for each integer $i$, the multiplication map $$F_i\t_{P} F_{m-i}\to F_m$$ is a perfect pairing of $P$-modules.
\end{definition}

\begin{example}The Koszul complex is the prototype of a $\DGGamma$-algebra which exhibits Poincar\'e duality. \end{example}

\section{Poincar\'e duality.}\label{PD}

\begin{data}\label{data30} Let $P$  
a commutative Noetherian ring and  
 $$F: \quad 0\to F_4\xrightarrow{d_4}
F_3\xrightarrow{d_3}F_2\xrightarrow{d_2}F_1\xrightarrow{d_1}F_0=P$$
be a length four resolution of a cyclic $P$-module 
by finitely generated free $P$ modules. Assume that $F_4$ has rank one.  Let $(-)^\vee$ denote the functor $\Hom_P(-,F_4)$. Assume also that the complexes $F$ and $F^\vee$ are isomorphic. \end{data}

\bigskip

The main result in the paper,  
Theorem~\ref{main}, 
 states that the resolution $F$ of Data~\ref{data30} is a $\DGGamma$-algebra 
 which exhibits Poincar\'e duality. In the present section, we focus on the Poincar\'e duality. We identify perfect pairings $$F_1\t_P F_3\to F_4\quad\text{and}\quad F_2\t_P F_2\to F_4$$which interact well with the differential of $F$.  The goal 
in the present section 
is to prove Lemma~\ref{Mar-12}. 

The most important feature of Lemma~\ref{Mar-12} is the homomorphisms
$\psi_3$ and $\psi_4$ which satisfy the differential properties (\ref{A.2}) and (\ref{A.3}) and which induce the perfect pairings of (\ref{perfectp}). 

The multiplication table which makes $F$ become a $\DGGamma$-algebra is given in Lemma~\ref{mult-table}. The maps $\psi_3$ and $\psi_4$ which appear in Lemma~\ref{mult-table} are directly imported from Lemma~\ref{Mar-12}, with no change. The multiplication $\Psi_1$ and $\Psi_2$ of Lemma~\ref{mult-table} are {\bf not the same} as the maps $\psi_1^\dagger$ and $\psi_2^\dagger$ of Lemma~\ref{Mar-12}. Indeed, in Sections~\ref{base case},
 \ref{Properties of the data.}, and
\ref{Inductive step.}, we completely ignore the maps $\psi_1^\dagger$ and $\psi_2^\dagger$ of Lemma~\ref{Mar-12}  while we construct a suitable map $\Psi_{1,2}$, as described in Lemma~\ref{2^n}. In Section~\ref{prove-3}, we modify the maps $\psi_1^\dagger$ and $\psi_2^\dagger$ from Lemma~\ref{Mar-12} in order to produce a suitable map $\Psi_{1,3}$, as describe in Lemma~\ref{3}. The ultimate map $\Psi_1$ of Lemma~\ref{mult-table} is a linear combination of $\Psi_{1,2}$ and $\Psi_{1,3}$; see (\ref{prove main}). The ultimate map $\Psi_2$ of  Lemma~\ref{mult-table} 
is defined in (\ref{3.3.1}) in terms of $\Psi_1$ and the perfect pairings of (\ref{perfectp}). We use the tag $^\dagger$ to indicate when the maps $\psi_1^\dagger$ and $\psi_2^\dagger$ of Lemma~\ref{Mar-12} are being used.

\begin{lemma}
\label{Mar-12} Adopt Data~{\rm\ref{data30}}.  
 Then
there exist homomorphisms 
\begin{align}\psi_1^\dagger&:\ts\bw^2 F_1\to F_2,&&& \psi_2^\dagger&:F_1\t F_2\to F_3,\label{3.2.1}\\\notag  \psi_3&:F_1\t F_3\to F_4, \text{ and} &&& \psi_4&:D_2F_2\to F_4\end{align}
which satisfy the following two conditions.
\begin{enumerate}[\rm(a)]
 \item\label{Mar-12.a}
The homomorphisms of {\rm(\ref{3.2.1})} satisfy the equations 
\begin{align}
\psi_3(x_1\t d_4(x_4))&{}=d_1(x_1)\cdot x_4, \label{A.2}\\
\psi_4(x_2\cdot d_3(x_3))&{}=-\psi_3(d_2(x_2)\t x_3),\label{A.3}\\
(d_2\circ \psi^\dagger_1)(x_1\w y_1)&{}=d_1(x_1)\cdot y_1-d_1(y_1)\cdot x_1,\label{new b}
\\
(d_3\circ \psi^\dagger_2)(x_1\t x_2)&{}=d_1(x_1)\cdot x_2-\psi^\dagger_1(x_1\w d_2(x_2)),
\label{b}\\
(d_4\circ \psi_3)(x_1\t x_3)&{}=d_1(x_1)\cdot x_3-\psi^\dagger_2(x_1\t d_3(x_3)),\text{ and}\label{c}
\\
(d_4\circ \psi_4)(x_2^{(2)})&{}=\psi^\dagger_2(d_2(x_2)\t x_2),\label{d}
\end{align} for all $x_i$ and $y_i$ in $F_i$.
\item \label{Mar-12.b} If 
the $P$-module homomorphisms 
\begin{equation}
\Phi_1:F_1\to F_3^\vee,\quad \Phi_2:F_2\to F_2^{\vee},\quad\text{and}\quad\Phi_3:F_3\to F_1^{\vee}, \label{perfectp}\end{equation}
are defined by$$\Phi_1(x_1)= \psi_3(x_1\t -),\quad \Phi_2(x_2)= \psi_4(x_2\cdot -),\quad \text{and}\quad\Phi_3(x_3)=\psi_3(-\t x_3),$$
respectively, with $x_i\in F_i$, then the homomorphisms of {\rm(\ref{perfectp})} are isomorphisms.\end{enumerate} 
\end{lemma}

We prove Lemma~\ref{Mar-12} in four steps. In Lemma~\ref{one} we obtain maps $\psi^\dagger_1$,  $\psi^\dagger_2$, $\psi_3$, and $\psi_4$ which satisfy all of the differential properties (\ref{A.2}) -- (\ref{d}). In Remark~\ref{two} we record the ramifications of modifying $\psi_3$ by a small homotopy. In Lemma~\ref{three} we modify $\psi_3$ in order to make the new version of $\psi_3:F_1\t F_3\to F_4$ be a perfect pairing. Finally, we complete the proof in 
\ref{four} by showing 
 that the map $\Phi_2$ of (\ref{perfectp}) is also an isomorphism.
 
We learned the technique that is used  in the proof of Lemma~\ref{one} from \cite[Prop. 1.1]{BE77}. This technique is similar to the Tate method of killing cycles \cite[Sect.~2]{T57}; one kills cycles of even degree with exterior variables and one kills cycles of odd degree with divided power variables.

\begin{lemma}
\label{one} Adopt Data~{\rm\ref{data30}}.  
Then there are $P$-module homomorphisms $\psi^\dagger_1$,  $\psi^\dagger_2$, $\psi_3$, and $\psi_4$, as described in 
{\rm(\ref{3.2.1})},
which satisfy {\rm(\ref{A.2})}, 
{\rm(\ref{A.3})}, {\rm(\ref{new b})}, {\rm(\ref{b})}, {\rm(\ref{c})}, and {\rm(\ref{d})}.  
\end{lemma}

\begin{proof} Consider the complex 
$$G: G_5\xrightarrow{g_5}  G_4\xrightarrow{g_4}G_3\xrightarrow{g_3}G_2\xrightarrow{g_2}G_2\xrightarrow{g_1}G_1\xrightarrow{g_1}G_0,$$
where $$G_i=\begin{cases} 
F_0&\text{if $i=0$}\\
F_1&\text{if $i=1$}\\
\bw^2F_1\p F_2&\text{if $i=2$}\\
(F_1\t F_2)\p F_3&\text{if $i=3$}\\
D_2F_2\p (F_1\t F_3)\p F_4&\text{if $i=4$}\\(F_2\t F_3)\p (F_1\t F_4)&\text{if $i=5$},\end{cases}$$
and 
$$g_5=
\bmatrix
1\t d_3&0\\
d_2\t 1&-1\t d_4\\
0&d_1\t 1
\endbmatrix,\ g_4=\bmatrix
d_2&-1\t d_3&0\\
0&d_1\t 1&d_4
 \endbmatrix,\  g_3=\bmatrix -1\t d_2&0\\
d_1\t 1&d_3\endbmatrix,$$ $g_2=\bmatrix d_1\w 1-1\w d_1& d_2\endbmatrix$, and  $g_1=d_1$.
The comparison theorem guarantees that there is a map of complexes $c:G\to F$ which extends the identity map in the first two components:
$$\xymatrix{
G_5\ar[d]\ar[r]^{
g_5}
&G_4\ar[d]^{c_4}\ar[r]^{g_4}&
G_3\ar[d]^{c_3}\ar[r]^{g_3}&
G_2\ar[d]^{c_2}\ar[r]^{g_2}&
G_1\ar[r]^{g_1}\ar[d]^{=}& G_0\ar[d]^{=}\\
0\ar[r]&F_4
\ar[r]^{d_4}&F_3
\ar[r]^{d_3}&F_2
\ar[r]^{d_2}&F_1
\ar[r]^{d_1}&F_0.
}$$
We name the interesting components of the $c_i$:
$$c_4=\bmatrix \psi_4& \psi_3&u_4\endbmatrix, \quad c_3=\bmatrix \psi^\dagger_2&u_3\endbmatrix, \quad\text{and}\quad 
c_2=\bmatrix \psi^\dagger_1&u_2\endbmatrix.$$
It is easy to check that the map $u_i$ can be taken to be $\id_{F_i}$ for $2\le i\le 4$. It is also   easy to read the properties {\rm(\ref{A.2})}, 
{\rm(\ref{A.3})}, {\rm(\ref{new b})}, {\rm(\ref{b})}, {\rm(\ref{c})}, and {\rm(\ref{d})} from the fact that $c$ is a map of complexes.
\end{proof}

\begin{remark}
\label{two} Adopt Data~{\rm\ref{data30}}. Let $\psi^\dagger_1$,  $\psi^\dagger_2$, $\psi_3$, and $\psi_4$, as described 
in 
{\rm(\ref{3.2.1})},
be 
$P$-module homomorphisms which 
satisfy {\rm(\ref{A.2})}, 
{\rm(\ref{A.3})}, {\rm(\ref{new b})}, {\rm(\ref{b})}, {\rm(\ref{c})}, and {\rm(\ref{d})}.
 A straightforward calculation shows that if
 $$\sigma: F_1\t F_2 \to F_4$$ is a $P$-module homomorphism, and 
\begin{align*}
\psi_1'&{}=\psi^\dagger_1,&&&
\psi_2'&{}=\psi^\dagger_2+d_4\circ \sigma,\\
\psi_3'&{}=\psi_3-\sigma\circ (1\t d_3),\text{ and}&&&
\psi_4'(x_2^{(2)})&{}=\psi_4(x_2^{(2)})+\sigma(d_2(x_2)\t x_2),\end{align*}
then the equations 
{\rm(\ref{A.2})}, 
{\rm(\ref{A.3})}, {\rm(\ref{new b})}, {\rm(\ref{b})}, {\rm(\ref{c})}, and {\rm(\ref{d})} are also 
satisfied when each $\psi_i$ or $\psi^\dagger_i$ is replaced with $\psi_i'$.
\end{remark}

Lemma~\ref{three} is a critical step in the present paper. If $P$ is local and $F$ is a minimal resolution, then $\psi_3$ is automatically a perfect pairing. (See, for example, \cite[Thm.~1.5]{BE77}.) In Lemma~\ref{three}, we prove that in the general situation (when $P$ is not necessarily local and $F$ is not necessarily a minimal resolution), it is possible to modify $\psi_3$ to make it become a perfect pairing. 
The perfect pairing $\psi_3$, and the corresponding statement for $\psi_4$,  are used throughout the paper. We define many maps into $F_3$ or $F_2$, by showing the value of the map after it is combined with $\psi_3$ or $\psi_4$. This style of definition is legitimate only because of assertion (\ref{Mar-12.b}) in Lemma~\ref{Mar-12}.

\begin{lemma}
\label{three}  Adopt Data~{\rm\ref{data30}} and  let $\psi^\dagger_1$,  $\psi^\dagger_2$, $\psi_3$, and $\psi_4$, as described 
in 
{\rm(\ref{3.2.1})}, 
be 
$P$-module homomorphisms which 
satisfy {\rm(\ref{A.2})}, 
{\rm(\ref{A.3})}, {\rm(\ref{new b})}, {\rm(\ref{b})}, {\rm(\ref{c})}, and {\rm(\ref{d})}.
Then there is a homomorphism $\sigma: F_1\t F_2\to F_4$ such that
 $$\psi_3 - \sigma \circ (1\t d_3):F_1\t F_3\to F_4$$
is a perfect pairing. \end{lemma} 

\begin{proof} Data~\ref{data30} guarantees that the complexes $F$ and $F^\vee$ are isomorphic. Let 
\begin{equation}\label{31.7.1}\xymatrix {F\ar[d]^{\phi}&
0\ar[r]&F_4\ar[r]^{d_4}\ar[d]^{\phi_4}&F_3\ar[r]^{d_3}\ar[d]^{\phi_3}&F_2\ar[r]^{d_2}\ar[d]^{\phi_2}&F_1\ar[r]^{d_1}\ar[d]^{\phi_1}&F_0\ar[d]^{\phi_0}\\F^\vee&
0\ar[r]&F_0^\vee\ar[r]^{d_1^\vee}&F_1^\vee\ar[r]^{d_2^\vee}&F_2^\vee\ar[r]^{d_3^\vee}&F_3^\vee\ar[r]^{d_4^\vee}&F_4^\vee,}\end{equation}
be one such isomorphism. The most natural isomorphism $F_4\to F_0^\vee$ is the map $\Phi_4$ which sends the element $x_4$ of $F_4$ to the homomorphism $P\to F_4$ which sends $x_0$ in $P$ to $x_0 x_4$ in $F_4$. In other words,
\begin{equation}\label{3.5.1.5}[\Phi_4(x_4)](x_0)=x_0 x_4,\end{equation}
for $x_0\in P$ and $x_4\in F_4$.
 Observe that there is a unit $u$ in $P$ with $u\phi_4=\Phi_4$. 
(Indeed, the source and target of $\phi_4$ are both isomorphic to $P$ and every $P$-module automorphism of $P$ is given by multiplication by a unit of $P$.) At any rate,
\begin{equation}\label{31.8.1}[(u\phi_4)(x_4)](x_0)=x_0 x_4,\end{equation} for $x_0\in P$ and $x_4\in F_4$.

Let $\rho: F_1\t F_3\to F_4$ be the homomorphism $$\rho(x_1\t x_3)=u[\phi_3(x_3)](x_1),$$ where $\phi_3:F_3\to F_1^\vee$ is the isomorphism of (\ref{31.7.1}) and $u$ is the unit of (\ref{31.8.1}). Observe that $\rho$ is a perfect pairing. We complete the proof by showing 
 that
there is a homomorphism $\sigma: F_1\t F_2\to F_4$ such that
 $$\psi_3 -\rho =\sigma \circ (1\t d_3).$$

The complex $(F_1\t_P F)^\vee$ is exact. It suffices  to show that $\psi_3-\rho$ is in $$\ker (1\t d_4)^\vee.$$ In other words, we prove that
\begin{equation}\label{31.9.1}\psi_3(x_1\t d_4(x_4))=\rho((x_1\t d_4(x_4)),\end{equation}
for $x_1\in F_1$ and $x_4\in F_4$. Observe that
\begin{align*} 
\rho(x_1\t d_4(x_4))={}&u[\phi_3(d_4(x_4))](x_1)\\
{}={}&u[(d_1^\vee\circ \phi_4)(x_4)](x_1),&&\text{by (\ref{31.7.1}),}\\
{}={}&u[\phi_4(x_4)](d_1(x_1))\\
{}={}&d_1(x_1)\cdot x_4,&&\text{by (\ref{31.8.1})},\\
{}={}&\psi_3(x_1\t d_4(x_4)),&&\text{by (\ref{A.2})}.
\end{align*}
Equation (\ref{31.9.1}) has been established and the proof is complete.
\end{proof}

\begin{chunk}
\label{four} {\it Proof of Lemma~{\rm\ref{Mar-12}}.} Apply Lemma~\ref{one} to find maps $\psi^\dagger_1$, $\psi^\dagger_2$, $\psi_3$, and $\psi_4$ which satisfy assertion (\ref{Mar-12.a}) of Lemma~{\rm\ref{Mar-12}}. Apply Lemma~\ref{three} to find a map $\sigma$ with $$\psi_3-\sigma\circ (1\t d_3):F_1\t F_3\to F_4$$ a perfect pairing. Use this $\sigma$ to modify the $\psi_i$ and $\psi^\dagger_i$, as described in Remark~\ref{two}. Reuse the old names and call the modified maps $\psi_i$ and $\psi^\dagger_i$. At this point assertion (\ref{Mar-12.a}) of Lemma~{\rm\ref{Mar-12}} holds and the maps $\Phi_1$ and $\Phi_3$ of (\ref{perfectp}) are isomorphisms. It remains to show that
the homomorphism $\Phi_2:F_2\to F_2^\vee$ of {\rm(\ref{perfectp})} is an also an isomorphism.
Define 
$\Phi_0:F_0\to F_4^\vee$ to be the map which sends $x_0\in F_0=P$ to ``multiplication  by $x_0$'' in $F_4^\vee=\Hom_P(F_4,F_4)$ and define $\Phi_4:F_4\to F_0^\vee$ 
to be the map which sends $x_4\in F_4$ to ``multiplication  by $x_4$'' in $$F_0^\vee=\Hom_P(F_0,F_4)=\Hom_P(P,F_4).$$A straightforward calculation shows  that
\begin{equation}
\notag\xymatrix{
0\ar[r] &F_4\ar[r]^{d_4}\ar[d]^{\Phi_4}&F_3 \ar[r]^{d_3}\ar[d]^{\Phi_3}&F_2\ar[r]^{d_2}\ar[d]^{-\Phi_2}&F_1\ar[r]^{d_1}\ar[d]^{\Phi_1}&F_0\ar[d]^{\Phi_0}\\
0\ar[r]&F_0^\vee\ar[r]^{d_1^\vee}&
F_1^\vee\ar[r]^{d_2^\vee}&
F_2^\vee\ar[r]^{d_3^\vee}&
F_3^\vee\ar[r]^{d_4^\vee}&F_4^\vee}\end{equation}
is a map of complexes. (A version of this calculation appears as the  proof of \cite[Thm.~1.5]{BE77}.) The map $\Phi_i$ is an isomorphism for $i$ equal to $0$, $1$, $3$, and $4$. It follows that $\Phi_2$ is also an isomorphism. \hfill \qed
\end{chunk}

\section{The main result.}
The main result in the paper is Theorem~\ref{main}. The proof  is based on three Lemmas. The proof of Lemma~\ref{mult-table} is given in Section~\ref{prove-mult-table}; the proof of Lemma~\ref{3} is given in
Section~\ref{prove-3}; and the proof of Lemma~\ref{2^n} is given in Section~\ref{Inductive step.}. Preliminary calculations that are used in Section~\ref{Inductive step.} are made in Sections \ref{base case} and \ref{Properties of the data.}.

\begin{data}\label{A.1} Let $P$ be a commutative Noetherian ring and
$$F:\quad 0\to F_4\xrightarrow{d_4}
F_3\xrightarrow{d_3}F_2\xrightarrow{d_2}F_1\xrightarrow{d_1}F_0=P$$
be an acyclic complex of free $P$-modules with $\rank F_4=1$.
Assume that $P$-module homomorphisms 
$$\psi_3:F_1\t F_3\to F_4 \quad  \text{and} \quad \psi_4:D_2F_2\to F_4$$
have been identified with 
\begin{align*}
&(\ref{A.2})&&&\psi_3(x_1\t d_4(x_4))&{}=d_1(x_1)\cdot x_4,\\
&(\ref{A.3})&&&\psi_4(x_2\cdot d_3(x_3))&{}=-\psi_3(d_2(x_2)\t x_3),
\end{align*}
for all $x_i$ 
in $F_i$. Let $(-)^\vee$ be the functor $\Hom_P(-,F_4)$.  Assume  that the $P$-module homomorphisms 
\begin{align*}
& &&F_1\to F_3^\vee,&&F_2\to F_2^{\vee},&\text{and}\quad&F_3\to F_1^{\vee},\intertext{which are given by}
(\ref{perfectp})& &&x_1\mapsto \psi_3(x_1\t -),&&x_2\mapsto \psi_4(x_2\cdot -),&\text{and}\quad&x_3\mapsto\psi_3(-\t x_3),\end{align*}
respectively, are isomorphisms.
\end{data}

\begin{Remark} 
Recall that Lemma~\ref{Mar-12} guarantees that the data of \ref{data30} gives rise to the data of \ref{A.1}. The information about $\psi_1^\dagger$ and $\psi_2^\dagger$ from Lemma~\ref{Mar-12} is not used in the present section. (This information is used in Section~\ref{prove-3}).
\end{Remark}

\begin{definition}\label{B.1} Adopt Data~{\rm\ref{A.1}}. Let 
$N$ be an integer. The $P$-module homomorphism 
$$\psi_1:F_1\t F_1\to F_2$$
is called 
{\it $N$-compatible with the data of {\rm\ref{A.1}}} if  
\begin{equation}\label{B.4}(d_2\circ \psi_1)(x_1\t y_1)= N\Big(d_1(x_1)\cdot y_1-d_1(y_1)\cdot x_1\Big)\quad\text{and}\end{equation}
\begin{equation}\label{B.5}\psi_4(\psi_1(x_1\t d_2(x_2))\cdot x_2)=N d_1(x_1)\cdot \psi_4( x_2^{(2)}), \end{equation}
for all $x_i$ and $y_i$ in $F_i$. 
\end{definition}
Recall that our usage of the phrase ``alternating map'' is explained in \ref{2.1}.(\ref{2.1.c}). 
\begin{lemma}\label{mult-table}
 Adopt the  data of {\rm \ref{A.1}}. 
Suppose that  $\Psi_1:T_2F_1\to F_2$ is an alternating map which is $1$-compatible with the data of {\rm \ref{A.1}} in the sense of Definition~{\rm\ref{B.1}}.
Define $\Psi_2:F_1\t F_2\to F_3$  by
\begin{equation}\label{3.3.1}\psi_3\Big(y_1\t \Psi_2(x_1\t x_2)\Big)=\psi_4\Big(\Psi_1(y_1\t x_1)\cdot x_2)\Big),\end{equation}
for $x_i$ and $y_i$ in $F_i$.
 Then $F$ has the structure of a $\DGGamma$-algebra with multiplication $$F\times F\to F$$ given by
\begingroup\allowdisplaybreaks\begin{align} x_1\times y_1&{}=\Psi_1(x_1\w y_1),\notag\\
x_1\times x_2=\phantom{-}x_2\times x_1&{}=\Psi_2(x_1\t x_2),\notag\\
x_1\times x_3=-x_3\times x_1&{}= \psi_3(x_1\t x_3),\notag\\
x_2\times y_2&{}= \psi_4(x_2\cdot y_2),\quad\text{and}\notag\\
x_2^{(2)}&{}=\psi_4(x_2^{(2)}),\label{3.3.2}
\end{align}\endgroup
for  $x_i$ and $y_i$ in $F_i$. In {\rm(\ref{3.3.2})},  the divided power on the left is computed in $F$ and the divided power on the right is computed in $D_{\bullet}F_2$.
\end{lemma}

\begin{lemma}\label{3} 
Adopt the  data of {\rm \ref{data30}}. 
Then there exists an alternating map $$\Psi_{1,3}:T_2F_1\to F_2$$ which is $3$-compatible with the data of {\rm \ref{A.1}}. 
\end{lemma}

\begin{lemma}\label{2^n}
Adopt the  data of {\rm \ref{A.1}}.
Then there exists an alternating map $$\Psi_{1,2}:T_2F_1\to F_2$$ which is $N$-compatible with the data of {\rm \ref{A.1}}, where $N=2^n$ is some positive power of two. 
\end{lemma}

\begin{theorem}
\label{main}
 Adopt the  data of {\rm \ref{data30}}. Then $F$ has the structure of a $\DGGamma$-algebra which exhibits Poincar\'e duality. 
\end{theorem}

Once Lemmas~\ref{mult-table}, \ref{3}, and \ref{2^n} are established in Sections~\ref{prove-mult-table},
 \ref{prove-3}, 
\ref{base case}, 
\ref{Properties of the data.}, and
\ref{Inductive step.}, then the proof of Theorem~\ref{main} follows readily.

\ms\noindent
{\it Proof of Theorem~{\rm\ref{main}}.} Fix homomorphisms $\psi_1^\dagger$, $\psi_2^\dagger$, $\psi_3$, and $\psi_4$ with all of the properties which are listed in Lemma~\ref{Mar-12}.
Lemma~\ref{Mar-12} guarantees that the data of \ref{data30} gives rise to the data of \ref{A.1}. Apply Lemmas \ref{3} and \ref{2^n} to obtain alternating maps $\Psi_{1,3}$ and $\Psi_{1,2}$ from $T_2F_1$ to $F_2$ which are $3$-compatible and $2^n$-compatible 
with the data of \ref{A.1}, respectively. There exist integers $a$ and $b$ with $3a+2^nb=1$. Let \begin{equation}\label{prove main}\Psi_1=a\Psi_{1,3}+b\Psi_{1,2}.\end{equation} It follows that $\Psi_1:T_2F_1\to F_2$ is an alternating map which is  $1$-compatible with the data of \ref{A.1}. The conclusion follows from Lemma~\ref{mult-table}. \hfill\qed

\begin{remark} Theorem~\ref{main} is already known if $P$ is Gorenstein and local and $F$ is a minimal resolution; see \cite{KM80,K87}. The purpose of the present paper is to remove these three unnecessary hypotheses. The proofs in \cite{KM80,K87} can probably be reconfigured in order to avoid the hypothesis that $P$ is Gorenstein. (In particular, the present paper does not depend on the results of the older papers and the argument in the present paper avoids this hypothesis without making any great effort.) The hypothesis that the resolution $F$ is minimal is clearly unnecessary. Indeed, if $P$ is local and $F$ is an arbitrary resolution of length four which is self-dual and has $F_0=P$, then $F$ is isomorphic to the direct sum of a minimal resolution plus a trivial resolution of the form
$$ 0\to E'\xrightarrow{\bmatrix 0\\1\endbmatrix} E\p E' \xrightarrow{\bmatrix 1&0\endbmatrix} E\to 0,$$ for some free $P$-modules $E$ and $E'$ of the same rank. (The right-most $E$ is in homological position one.)  It does no harm if we write $E^\vee=\Hom_P(E,F_4)$ in place of $E'$. It is not difficult to extend the $\DGGamma$-structure from the minimal resolution to $F$; see Example~\ref{Apr22}.

 The serious work in the proof of Theorem~\ref{main} is involved in removing the hypothesis ``local''. Of course, once Theorem~\ref{main} is completely established, then the result holds for all $\mathbb Z$-algebras. We wonder how generally the statement
$$ \text {$F_\mfp$ is a $\DGGamma$-algebra for all prime ideals $\mfp$}\implies \text{$F$ is a $\DGGamma$-algebra} $$holds.
\end{remark}

\begin{example}\label{Apr22} Let $P$ be a commutative Noetherian ring,
$$F:\quad 0\to F_4\xrightarrow{d_4}F_3\xrightarrow{d_3}F_2\xrightarrow{d_2} F_1\xrightarrow{d_1}F_0=P$$ be  a $\DGGamma$-algebra resolution by free $P$-modules  which exhibits Poincar\'e duality, and $E$ be a free $P$-module of finite rank.
Let $G$ be the complex $$G:\quad 0\to F_4\xrightarrow{g_4=\bmatrix d_4\\0\endbmatrix} \begin{matrix}F_3\\\p\\ E^\vee\end{matrix} \xrightarrow {g_3= \bmatrix d_3&0\\0&0\\0&1\endbmatrix }  \begin{matrix}F_2 \\\p\\ E\\\p\\E^\vee\end{matrix} \xrightarrow
{g_2=\bmatrix d_2&0&0\\0&1&0\endbmatrix}  \begin{matrix}F_1 \\\p\\ E\end{matrix} \xrightarrow{g_1=\bmatrix d_1&0\endbmatrix} F_0=P$$Then it is not difficult to check that the multiplication\begingroup\allowdisplaybreaks\begin{align*}
&G_1\t G_1\to G_2: &\bmatrix x_1\\e\endbmatrix \times  \bmatrix x_1'\\e'\endbmatrix &=\bmatrix x_1\cdot x_1'\\d_1(x_1)\cdot e'-d_1(x_1')\cdot e\\0\endbmatrix,
\\&\begin{matrix}G_2\t G_1\to G_3:\\
G_1\t G_2\to G_3:\end{matrix} &  \bmatrix x_2\\e'\\\e\endbmatrix\times \bmatrix x_1\\e\endbmatrix
=\bmatrix x_1\\e\endbmatrix \times  \bmatrix x_2\\e'\\\e\endbmatrix&=
\bmatrix x_1\cdot x_2-d_4(\e(e))\\d_1(x_1)\cdot \e\endbmatrix,
\\
&\begin{matrix}G_3\t G_1\to G_4:\\
G_1\t G_3\to G_4:\end{matrix} & 
  -\bmatrix x_3\\ \e\endbmatrix \times \bmatrix x_1\\e\endbmatrix
=\bmatrix x_1\\e\endbmatrix \times  \bmatrix x_3\\ \e\endbmatrix&= x_1\cdot x_3+\e(e), 
\\
&G_2\t G_2\to G_4:&\bmatrix x_2\\e\\\e\endbmatrix \times  \bmatrix x_2'\\e'\\\e'\endbmatrix&=x_2\cdot x_2'-\e'(e)-\e(e'),\text{ and}
\\
&D_2G_2\to G_4:&\bmatrix x_2\\e\\\e\endbmatrix^{(2)}&=x_2^{(2)}-\e(e)
\end{align*}\endgroup
gives $G$ the structure of a $\DGGamma$-algebra which exhibits Poincar\'e duality, for $x_i$ and $x_i'$ in $F_i$, $e$ and $e'$ in $E$, and $\ep$ and $\ep'$ in $E^\vee=\Hom_P(E,F_4)$.  
\end{example}

\section{The proof of Lemma \ref{mult-table}.}\label{prove-mult-table}
A version of this proof may also be found in  \cite{KM80}.

\begin{proof}The map $\Psi_1:T_2F_1\to F_2$ is an alternating map; consequently, we  write $\Psi_1(x_1\w y_1)$ instead of $\Psi_1(x_1\t y_1)$. It is clear that the proposed multiplication is graded-commutative. We demonstrate the differential, divided power, and associative properties
\begin{align*}
\operatorname{D}_{i,j}:& \quad d_{i+j}(x_i\times x_j)=d_i(x_i)\times x_j+(-1)^i x_i\times d_j(x_j),\\
\operatorname{DP}:&\quad d_4(x_2^{(2)})=d_2(x_2)\times x_2,\quad\text{and}\\
\operatorname{A}_{i,j,k}:&\quad (x_i\times x_j)\times x_k=x_i\times (x_j\times x_k),\end{align*}
with $x_\ell\in F_\ell$, for all relevant $i,j,k$.

\ms \noindent Property $\operatorname{D}_{1,1}$ is a consequence of (\ref{B.4}) because $N=1$.

\ms \noindent Property $\operatorname{D}_{1,2}$ is equivalent to
$$d_1(x_1)\cdot \psi_4(x_2\cdot y_2)- \psi_4\Big(\Psi_1(x_1\w d_2(x_2))\cdot y_2\Big)=\psi_4\Big(d_3(\Psi_2(x_1\t x_2))\cdot y_2\Big).$$
Observe that
\begingroup\allowdisplaybreaks\begin{align*}
&\phantom{{}+{}}\psi_4\Big((d_3\circ\Psi_2)(x_1\t x_2)\cdot y_2\Big)\\={}&-\psi_3\Big(d_2(y_2)\t \Psi_2(x_1\t x_2)\Big),&&\text{by (\ref{A.3}),}\\
{}={}&-\psi_4\Big(\Psi_1\big(d_2(y_2)\w x_1\big)\cdot x_2\Big),&&\text{by (\ref{3.3.1}),} \\
{}={}&d_1(x_1)\cdot \psi_4(x_2\cdot y_2)-\psi_4\Big(\Psi_1\big(x_1\w d_2(x_2)\big)\cdot y_2\Big),&&\text{by (\ref{B.5})}.
\end{align*}\endgroup

\ms \noindent Property $\operatorname{D}_{1,3}$ is equivalent to
$$d_1(x_1)\cdot \psi_3(y_1\t x_3) -\psi_3\Big(y_1\t \Psi_2(x_1\t d_3(x_3))\Big)=\psi_3\Big(y_1\t d_4(\psi_3(x_1\t x_3))\Big).$$

Observe that 
\begin{align*}
&\psi_3\Big(y_1\t \Psi_2\big(x_1\t d_3(x_3)\big)\Big)\\
={}&\psi_4\Big(\Psi_1(y_1\w x_1)\cdot d_3(x_3)\Big),&&\text{by (\ref{3.3.1}),}\\
={}&-\psi_3\Big((d_2\circ \Psi_1)(y_1\w x_1)\t x_3\Big),&&\text{by (\ref{A.3}),}\\
={}&-d_1(y_1)\cdot \psi_3(x_1\t x_3)
+d_1(x_1)\cdot \psi_3(y_1\t x_3),
&&\text{by (\ref{B.4}).}\end{align*}
On the other hand, 
$$\psi_3\Big(y_1\t d_4(\psi_3(x_1\t x_3))\Big)=d_1(y_1)\cdot \psi_3(x_1\t x_3)$$ by (\ref{A.2}).

\ms \noindent Property $\operatorname{D}_{1,4}$ is a consequence of (\ref{A.2}).

\ms \noindent Property $\operatorname{D}_{2,2}$ is equivalent to
$$\psi_3\Big(x_1\t d_4(\psi_4(x_2\cdot y_2))\Big)=\begin{cases}
\phantom{+}\psi_3\Big(x_1\t\Psi_2(d_2(x_2)\t y_2)\Big)\\
+\psi_3\Big(x_1\t\Psi_2(d_2(y_2)\t x_2)\Big).\end{cases}$$
Observe that 
\begin{align*}&\psi_3\Big(x_1\t\Psi_2(d_2(x_2)\t y_2)\Big)
+\psi_3\Big(x_1\t\Psi_2(d_2(y_2)\t x_2)\Big)\\
{}={}&\psi_4\Big(\Psi_1(x_1\w d_2(x_2))\cdot y_2\Big)
+
\psi_4\Big(\Psi_1(x_1\w d_2(y_2))\cdot x_2\Big),&&\text{by (\ref{3.3.1})},\\
{}={}&d_1(x_1)\cdot \psi_4(x_2\cdot y_2),&&\text{by (\ref{B.5})},\\
{}={}&\psi_3\Big(x_1\t d_4(\psi_4(x_2\cdot y_2))\Big),&&\text{by (\ref{A.2}).}
\end{align*}
 
\ms \noindent Property $\operatorname{DP}$ is equivalent to
$$\psi_3\Big(x_1\t d_4(\psi_4(x_2^{(2)}))\Big)=\psi_3\Big(x_1\t\Psi_2(d_2(x_2)\t x_2)\Big).$$
Observe that 
\begingroup\allowdisplaybreaks\begin{align*}\psi_3\Big(x_1\t\Psi_2(d_2(x_2)\t x_2)\Big)
{}={}&\psi_4\Big(\Psi_1(x_1\w d_2(x_2))\cdot x_2\Big)
,&&\text{by (\ref{3.3.1})},\\
{}={}&d_1(x_1)\cdot \psi_4(x_2^{(2)}),&&\text{by (\ref{B.5})},\\
{}={}&\psi_3\Big(x_1\t d_4(\psi_4(x_2^{(2)}))\Big),&&\text{by (\ref{A.2}).}
\end{align*}\endgroup

\ms \noindent Property $\operatorname{D}_{2,3}$ is a consequence of (\ref{A.3}).

\ms \noindent Property $\operatorname{A}_{1,1,2}$ is a consequence of (\ref{3.3.1}).

\ms \noindent Property $\operatorname{A}_{1,1,1}$ is equivalent to
$$\psi_3\Big(w_1\t \Psi_2(x_1\t \Psi_1(y_1\w z_1))\Big)=\psi_3\Big(w_1\t \Psi_2(z_1\t \Psi_1(x_1\w y_1))\Big),$$ for all $x_1$, $y_1$, $z_1$, and $w_1$ in $F_1$.

Apply (\ref{3.3.1}) twice. It suffices to show that
$$\psi_4(\Psi_1(w_1\w x_1)\cdot \Psi_1(y_1\w z_1))=\psi_4(\Psi_1(w_1\w z_1)\cdot \Psi_1(x_1\w y_1)).$$

Let $A:F_1\t T_3(F_1)\to F_4$ be the homomorphism
\begin{align*}&A(w_1\t x_1\t y_1\t z_1)\\{}={}&\psi_4(\Psi_1(w_1\w x_1)\cdot \Psi_1(y_1\w z_1))-\psi_4(\Psi_1(w_1\w z_1)\cdot \Psi_1(x_1\w y_1)).\end{align*}
We will prove that $A$ is identically zero.
Observe that \begingroup\allowdisplaybreaks\begin{align*}
&A(d_2(w_2)\t x_1\t y_1\t z_1)\\
={}& \psi_4(\Psi_1(d_2(w_2)\w x_1)\cdot \Psi_1(y_1\w z_1))-\psi_4(\Psi_1(d_2(w_2)\w z_1)\cdot \Psi_1(x_1\w y_1))\\
{}={}& -\psi_4(\Psi_1(x_1\w d_2(w_2))\cdot \Psi_1(y_1\w z_1))+\psi_4(\Psi_1(z_1\w d_2(w_2))\cdot \Psi_1(x_1\w y_1))\\
{}={}&\begin{cases}\phantom{+}
\psi_4\Big(\Psi_1\big(x_1\w (d_2\circ \Psi_1)(y_1\w z_1)\big)\cdot w_2\Big)-d_1(x_1)\cdot \psi_4\Big(w_2\cdot \Psi_1(y_1\w z_1)\Big)\\
-\psi_4\Big(\Psi_1\big(z_1\w (d_2\circ \Psi_1)(x_1\w y_1)\big)\cdot w_2\Big)+d_1(z_1)\cdot \psi_4\Big(w_2\cdot \Psi_1(x_1\w y_1)\Big),
\end{cases}&&\intertext{by (\ref{B.5}),}
{}={}&\begin{cases}\phantom{+}
 d_1(y_1)\cdot \psi_4\Big(\Psi_1(x_1\w z_1)\cdot w_2\Big)
-d_1(z_1)\cdot \psi_4\Big(\Psi_1(x_1\w y_1)\cdot w_2\Big)\\
-d_1(x_1)\cdot \psi_4\Big(w_2\cdot \Psi_1(y_1\w z_1)\Big)\\
-d_1(x_1)\cdot \psi_4\Big(\Psi_1(z_1\w y_1)\cdot w_2\Big)
+d_1(y_1)\cdot \psi_4\Big(\Psi_1(z_1\w x_1)\cdot w_2\Big)\\
+d_1(z_1)\cdot \psi_4\Big(w_2\cdot \Psi_1(x_1\w y_1)\Big)
\end{cases}\\
{}={}&0.\end{align*}\endgroup The next-to-last equality is due to (\ref{B.4}).
Apply Remark~\ref{R1} to see that there exists a homomorphism $\widetilde{A}:T_3(F_1)\to F_4$, with
$$d_1(w_1)\cdot \widetilde{A}(x_1\t y_1\t z_1)=A(w_1\t  x_1\t y_1\t z_1),$$for $x_1,y_1,z_1,w_1\in F_1$. 
Notice, in particular, that 
\begin{align}\label{21.20.4}d_1(x_1)\cdot \widetilde{A}(x_1\t y_1\t z_1)={}&-\psi_4\Big(\Psi_1(x_1\w z_1)\cdot \Psi_1(x_1\w y_1)\Big)\\{}={}&d_1(z_1)\cdot\widetilde{A}(x_1\t x_1\t y_1)\notag &&\text{and}\\
\label{21.20.5} d_1(x_1)\cdot \widetilde{A}(x_1\t x_1\t y_1)={}&0.\end{align} 
The ideal $\im d_1$ has positive grade; hence there is an element $u_1$ in $F_1$ with $d_1(u_1)$ a regular element of $P$. Assume that $x_1\in F_1$ has the property that
 $d_1(x_1)$ is a regular element of $P$. The $P$-modules $F_4$ and $P$ are isomorphic; so, $d_1(x_1)$ is also regular on $F_4$. It follows from (\ref{21.20.4}) and (\ref{21.20.5}) that $\widetilde{A}(x_1\t y_1\t z_1)$ is zero for all $y_1,z_1\in F_1$. Furthermore, 
 \begin{align*}&d_1(x_1)\cdot \widetilde{A}(w_1\t z_1\t y_1)\\{}={}&\psi_4(\Psi_1(x_1\w w_1)\cdot \Psi_1(z_1\w y_1))-\psi_4(\Psi(x_1\w y_1)\cdot \Psi_1(w_1\w z_1))\\{}={}&d_1(w_1)\cdot \widetilde{A}(x_1\t y_1\t z_1)=0.\end{align*}
Thus, $\widetilde{A}(w_1\t z_1\t y_1)=0$ for all $w_1,z_1,y_1$ in $F_1$ and $A$ is identically zero.  \end{proof}

Remark~\ref{R1} is obvious. The only complication is that one must think for a moment before  arranging the data in the right order. We use this argument often and for that reason we record an explicit statement and proof.
\begin{remark}
\label{R1} 
Adopt the Data of {\rm \ref{A.1}}. If $Y$ is a free $P$-module and $\phi:Y\t F_1\to F_4$ is a homomorphism with $\phi(y\t d_2(x_2))=0$ for all $y\in Y$ and $x_2\in F_2$, then there exists a homomorphism $\widetilde{\phi}:Y\to F_4$ with \begin{equation}\label{R1.1}d_1(x_1)\cdot \widetilde{\phi}(y)=\phi(y\t x_1)\end{equation} for all $y\in Y$ and $x_1\in F_1$. \end{remark}

\begin{proof} The complex $\Hom(Y,F^\vee)$ is acyclic. The hypothesis guarantees that
\begin{align*}y\mapsto \phi(y\t -)\in{}& \ker\Big(d_2^\vee:\Hom(Y,F_1^\vee)\to \Hom(Y,F_2^\vee)\Big)\\
{}={}&\im\Big(d_1^\vee:\Hom(Y,F_0^\vee)\to \Hom(Y,F_1^\vee)\Big).\end{align*}Thus, there is a homomorphism $\phi':Y\to F_0^\vee$ with 
$$\phi(y\t x_1)=[\phi'(y)]\big(d_1(x_1)\big),$$ for $y\in Y$ and $x_1\in F_1$. Of course, $$[\phi'(y)]\big(d_1(x_1)\big)= d_1(x_1)\cdot[\phi'(y)](1).$$ Define 
$\widetilde{\phi}:Y\to F_4$ to be the homomorphism  $\widetilde{\phi}(y)=[\phi'(y)](1)$.
We have established that equation (\ref{R1.1}) holds.\end{proof}

\section{The proof of Lemma~\ref{3}.}\label{prove-3}
The proof of Lemma~\ref{3} closely follows the idea of the proof in \cite{K87}.

\begin{data}
\label{data6}Adopt the notation of Data~\ref{data30} and fix homomorphisms $\psi^\dagger_1$, $\psi^\dagger_2$, $\psi_3$, and $\psi_4$ with all of the properties which are listed in Lemma~\ref{Mar-12}. 
Define $\a:\bw^2F_1\to F_2$ by
\begin{equation}\psi_4\Big(x_2\cdot \a(x_1\w x_1')\Big) =\begin{cases}
\phantom{+}2\psi_4([\psi^\dagger_1(x_1\w x_1')]\cdot x_2)\\
+\psi_3(x_1'\t \psi^\dagger_2(x_1\t x_2))\\
-\psi_3(x_1\t \psi^\dagger_2(x_1'\t x_2)).\end{cases}\label{alpha}\end{equation}\end{data}

\begin{lemma}
\label{Mar-13} Adopt the data of {\rm\ref{data6}}.
\label{22.7.1} Then there exists a map $\ts\g: \bw^2F_1\to F_3$ such that 
$$\psi_3\Big(d_2x_2\t \g(x_1\w x_1')\Big)+ \psi_4\Big(x_2\cdot \a(x_1\w x_1')\Big)=0
,$$ for $x_1,x_1'\in F_1$ and $x_2\in F_2$.
\end{lemma}

\begin{proof}Observe first that
$d_2\circ \a$ is identically zero.
Indeed, if $x_1,x_1'\in F_1$ and $x_3\in F_3$, then 
\begingroup \allowdisplaybreaks
\begin{align*}&
\phantom{{}-{}}\psi_3\Big((d_2\circ \a)(x_1\w x_1')\t x_3\Big)
\\{}={}& 
-\psi_4\Big(d_3(x_3)\cdot \a(x_1\w x_1')\Big),&&\text{by (\ref{A.3})},\\
{}={}& 
-\begin{cases}
\phantom{+}2\psi_4\Big([\psi^\dagger_1(x_1\w x_1')]\cdot d_3(x_3)\Big)\\
+\psi_3\Big(x_1'\t\psi^\dagger_2\big(x_1\t d_3(x_3)\big)\Big)\\
-\psi_3\Big(x_1\t \psi^\dagger_2\big(x_1'\t d_3(x_3)\big)\Big),\end{cases}&&\text{by (\ref{alpha}),}
\\
{}={}&-\begin{cases}
-2\psi_3\Big((d_2\circ\psi^\dagger_1)(x_1\w x_1')\t x_3\Big)\\
+\psi_3\Big(x_1'\t [d_1(x_1)\cdot x_3-(d_4\circ \psi_3)(x_1\t x_3))]\Big)\\
-\psi_3\Big(x_1\t [d_1(x_1')\cdot x_3-(d_4\circ \psi_3)(x_1'\t x_3))]\Big),
\end{cases}&&\text{by (\ref{A.3}) and (\ref{c}),}\\
{}={}&-\begin{cases}
-2d_1(x_1)\cdot\psi_3( x_1'\t x_3)\\
+2d_1(x_1')\cdot\psi_3( x_1\t x_3)\\
+
d_1(x_1)\cdot \psi_3(x_1'\t x_3)-d_1(x_1')\cdot\psi_3(x_1\t x_3)\\
-d_1(x_1')\cdot \psi_3(x_1\t x_3)+d_1(x_1)\cdot\psi_3(x_1'\t x_3).
\end{cases}\end{align*}\endgroup
The last equality is due to
(\ref{new b}) and (\ref{A.2}).
It is now clear that 
$$\psi_3((d_2\circ \a)(x_1\w x_1')\t x_3)=0$$ and
$d_2\circ \a$ is identically zero.

The complex $\Hom(\bw^2F_1, F)$ is acyclic; so there exists $\gamma:\bw^2F_1\to F_3$ such that
\begin{equation}d_3\circ \g=\a.\label{Mar-13.2}\end{equation} 
Observe that
\begin{align*}\psi_3\Big(d_2(x_2)\t \g(x_1\w x_1')\Big)={}&-\psi_4\Big(x_2\cdot (d_3\circ\g)(x_1\w x_1')\Big),&&\text{by (\ref{A.3}),}\\{}={}&{}-\psi_4\Big(x_2\cdot \a(x_1\w x_1')\Big),&&\text{by (\ref{Mar-13.2}).}\end{align*}
\vskip-28pt\end{proof}
 
\begin{lemma}
\label{5.2} Adopt the notation of {\rm\ref{data6}} and Lemma~{\rm\ref{Mar-13}}. Then
\begin{equation}\label{22.7.2}
3d_1(x_1)\cdot \psi_4(x_2^{(2)})=\psi_4((3\psi^\dagger_1-d_3\circ \g)(x_1\w d_2 x_2)\cdot x_2).\end{equation}
\end{lemma}
\begin{proof}
Observe that
\begingroup\allowdisplaybreaks\begin{align*}&-\psi_4\Big((d_3\circ \g)\big(x_1\w d_2(x_2)\big)\cdot x_2\Big)\\
{}={}& 
+\psi_3\Big(d_2(x_2)\t \g\big(x_1\w d_2(x_2)\big)\Big),
&&\text{by (\ref{A.3})},\\
{}={}& -\psi_4\Big(x_2\cdot \a\big(x_1\w d_2(x_2)\big)\Big),&&\text{by Lemma~\ref{22.7.1},}\\
{}={}&\begin{cases} 
-2\psi_4\Big(\big[\psi^\dagger_1\big(x_1\w d_2(x_2)\big)\big]\cdot x_2\Big)\\
-\psi_3\Big(d_2(x_2)\t \psi^\dagger_2(x_1\t x_2)\Big)\\
+\psi_3\Big(x_1\t \psi^\dagger_2\big(d_2(x_2)\t x_2\big)\Big),
\end{cases}&&\text{by (\ref{alpha}),}\\
{}={}&\begin{cases} 
-2\psi_4\Big(\big[\psi^\dagger_1\big(x_1\w d_2(x_2)\big)\big]\cdot x_2\Big)\\
+\psi_4\Big(x_2\cdot (d_3\circ\psi^\dagger_2)(x_1\t x_2)\Big)\\
+\psi_3\Big(x_1\t (d_4\circ\psi_4)(x_2^{(2)})\Big),
\end{cases}&&\text{by (\ref{A.3}) and (\ref{d})},\\
{}={}&\begin{cases} 
-2\psi_4\Big(\big[\psi^\dagger_1\big(x_1\w d_2(x_2)\big)\big]\cdot x_2\Big)\\
+\psi_4\Big(x_2\cdot \big[d_1(x_1)\cdot x_2-\psi^\dagger_1\big(x_1\w d_2(x_2)\big)\big]\Big)\\
+d_1(x_1)\cdot \psi_4(x_2^{(2)}),
\end{cases}&&\text{by (\ref{b}) and (\ref{A.2})},\\
{}={}&\begin{cases}-3\psi_4\Big(\psi^\dagger_1\big(x_1\w d_2 (x_2)\big)\cdot x_2\Big)\\ 
+3d_1(x_1)\cdot\psi_4(x_2^{(2)}).
\end{cases}
\end{align*}\endgroup
The final equality makes use of the fact that 
$x_2\cdot x_2=2x_2^{(2)}$ in $D_2(F_2)$. One concludes that
$$\psi_4\Big((3\psi^\dagger_1-d_3\circ \g)(x_1\w d_2 (x_2))\cdot x_2\Big)=3d_1(x_1)\cdot\psi_4(x_2^{(2)}),$$ which is (\ref{22.7.2}).
\end{proof}

\begin{chunk}
{\it Proof of Lemma~{\rm\ref{3}.}}
Let $\Psi_{1,3}:\bw^2F_1\to F_2$ be the alternating map $3\psi_1^\dagger-d_3\circ \gamma$. Take $N=3$. It is clear that 
(\ref{B.4}) is satisfied and it is shown in Lemma~\ref{5.2} that (\ref{B.5}) is also satisfied.
\hfill \qed
\end{chunk}

\section{The base case for Lemma~\ref{2^n}.}\label{base case} We use a complicated inductive modification process to prove Lemma~\ref{2^n}. Preliminary calculations for this modification process are made in Section~\ref{Properties of the data.}; the process itself is carried out in Section~\ref{Inductive step.}. In the present section we establish the base case. This argument is straightforward.

\begin{observation} 
\label{Y22.5}Adopt Data {\rm\ref{A.1}}. Then there  exists a $P$-module homomorphism  $$\ts \psi_1^{\la 0\ra}:T_2F_1 \to F_2$$ which is $2$-compatible with the data of {\rm\ref{A.1}} in the sense of Definition~{\rm \ref{B.1}}. \end{observation}

\begin{proof}
The complex $\Hom(T_2F_1,F)$ is exact and the homomorphism $$(d_1\t 1-1\t d_1):T_2F_1\to F_1$$is in the kernel of $d_{1*}$. Consequently, there is a homomorphism $\psi_1':T_2F_1\to F_2$  which satisfies (\ref{B.4}) with $N=1$.
Define $\rho_1:F_1\t T_2F_2\to F_4$ by
\begin{equation}\label{C.1}\rho_1(x_1\t x_2\t y_2)=\begin{cases}\phantom{+}\psi_4\Big(\psi_1'(x_1\t d_2(x_2))\cdot y_2\Big)\\
+\psi_4\Big(\psi_1'(x_1\t d_2(y_2))\cdot x_2\Big)\\
-d_1(x_1)\cdot \psi_4\Big(x_2\cdot y_2\Big).\end{cases}\end{equation}
Observe that 
\begin{align*}
&\phantom{-}\rho_1(x_1\t d_3(x_3)\t y_2)\\
{}={}&\phantom{-}\psi_4\Big(
\psi_1'(x_1\t d_2(y_2))\cdot d_3(x_3)
-d_1(x_1)\cdot d_3(x_3)\cdot y_2\Big)\\
{}={}&-\psi_3\Big(
(d_2\circ\psi_1')(x_1\t d_2(y_2))\t x_3)
-d_1(x_1)\cdot d_2(y_2)\t x_3\Big),&&\text{by (\ref{A.3}),}\\
{}={}&-\psi_3\Big(
d_1(x_1)\cdot  d_2(y_2)\t x_3)
-d_1(x_1)\cdot d_2(y_2)\t x_3\Big),&&\text{by (\ref{B.4}),}\\{}={}&\phantom{-}0.\end{align*}The homomorphism $\rho_1$ is symmetric in its second and third arguments; so,
$$\rho_1(-\t-\t d_3(-)):F_1\t F_2\t F_3\to F_4$$ is also identically zero. It follows from Lemma~\ref{write me} that there is a homomorphism 
$$\rho_2: T_3(F_1)\to F_4$$ with
\begin{equation}\rho_1(x_1\t x_2\t y_2)=\rho_2(x_1\t d_2(x_2)\t d_2(y_2)).\label{YD12-10-7-25}
\end{equation}Define $$\rho_3:T_2(F_1)\to F_3$$ by
\begin{equation}\label{YD12-10-7-20}\psi_3(x_1''\t \rho_3(x_1\t x_1'))= \rho_2(x_1\t x_1'\t x_1'').\end{equation}

Observe that 
\begin{equation}\label{C.2}\psi_4\Big(x_2\cdot (d_3\circ \rho_3)(x_1\t d_2(x_2))\Big)=-\rho_1(x_1\t x_2\t x_2).\end{equation}
Indeed,
\begin{align*}
& \phantom{{}-{}}\psi_4\Big(x_2\cdot (d_3\circ \rho_3)(x_1\t d_2(x_2))\Big)\\
{}={}&-\psi_3\Big(d_2(x_2)\t  \rho_3(x_1\t d_2(x_2))\Big),&&\text{by (\ref{A.3}),}\\
{}={}&-\rho_2(x_1\t d_2(x_2)\t d_2(x_2)),&&\text{by (\ref{YD12-10-7-20}),}\\
{}={}&-\rho_1(x_1\t x_2\t x_2),&&\text{by (\ref{YD12-10-7-25}).}
\end{align*} 
 Define $\psi_1^{\la 0\ra}:T_2(F_1)\to F_2$ to be the homomorphism \begin{equation}\label{2.3.4} \psi^{\la 0\ra}_1=2\psi_1'+ d_3\circ\rho_3.\end{equation}
Observe that 
  $d_2\circ \psi_1^{\la 0\ra}=2d_2\circ \psi_1'= 2(d_1\t 1-1\t d_1)$,  
and 
\begin{align*}&\psi_4(\psi_1^{\la 0\ra}(x_1\t d_2(x_2))\cdot x_2)\\
{}={}&2\psi_4(\psi_1'(x_1\t d_2(x_2))\cdot x_2)+ 
\psi_4((d_3\circ\rho_3)(x_1\t d_2(x_2))\cdot x_2),&&\text{by (\ref{2.3.4}),}
\\
{}={}&
2\psi_4(\psi_1'(x_1\t d_2(x_2))\cdot x_2)
-\rho_1(x_1\t x_2\t x_2),&&\text{by (\ref{C.2}),}
\\
{}={}&\begin{cases}
\phantom{+}2\psi_4\Big(\psi_1'(x_1\t d_2(x_2))\cdot x_2\Big)\\
-2\psi_4\Big(\psi_1'(x_1\t d_2(x_2))\cdot x_2\Big)
-d_1(x_1)\cdot \psi_4(x_2\cdot x_2),
\end{cases}&&\text{by (\ref{C.1}),}
\\
{}={}&2d_1(x_1)\cdot \psi_4(x_2^{(2)}).
\end{align*} Thus, $\psi_1^{\la 0\ra}$ is $2$-compatible with the data of \ref{A.1}. \end{proof} 

The following fact was used in the proof of Observation~\ref{Y22.5}. It is established by a
 routine diagram chase. More general statements are also true.\begin{lemma}
\label{write me}Adopt the data of {\rm\ref{A.1}}. If $G$ is a free $P$-module, then the complex 
$$(G\t F_1\t F_1)^\vee\xrightarrow{(1\t d_2\t d_2)^\vee} (G\t F_2\t F_2)^\vee\xrightarrow{\bmatrix (1\t d_3\t 1)^\vee\\(1\t 1\t d_3)^\vee\endbmatrix} \begin{matrix} (G\t F_3\t F_2)^\vee\\\p\\
(G\t F_2\t F_3)^\vee\end{matrix} $$ is exact.
\end{lemma}
\begin{proof} Each row and column of the double complex $(G\t F\t F)^\vee$ is acyclic. 
\end{proof} 

Lemma~\ref{2.2} is the final piece of the base case. This result is a consequence of the prime avoidance lemma.
\begin{lemma}\label{2.2}
Let $P$ be a commutative Noetherian ring, $F_1$ be a free $R$-module, and $d_1:F_1\to P$ be a $P$-module homomorphism. If $\im d_1$ has positive grade, 
then there exists an element $g_1$ in $F_1$ with $Pg_1$ a summand of $F_1$ and $d_1(g_1)$ a regular element of $P$.\end{lemma}

\begin{proof} Let
 $$S=\{\mfp\in \Ass (P)\mid \text{ $\mfp$ is not properly contained in $\mfq$ for any $\mfq\in \Ass(P)$}\}.$$
The point is that the set of zero divisors of $P$ is $\cup_{\mfp\in S} \mfp$ and no prime of $S$ contains another prime of $S$. The hypothesis that $\im d_1$ has positive grade ensures that \begin{equation}\label{26.6.1}\im d_1 \not \subseteq \bigcup_{\mfp\in S}\mfp.\end{equation} Let $F_1=Pg\p F_1'$ be a decomposition of $F_1$ into free submodules with $\rank Pg$ equal to $1$. We will identify an element $\theta$ of $F_1'$ such that
\begin{equation}d_1(g+\theta)\notin \bigcup_{\mfp\in S}\mfp.\label{26.6.2}\end{equation} Once $\theta$ has been identified, then we define  $g_1$ to be $g+\theta$ and we observe that $g_1$ generates a summand of $F_1$ and $d_1(g_1)$ is a regular element of $P$.

Decompose $S$ into two subsets:
$$S_0=\{\mfp\in S\mid d_1(g)\in \mfp\}\quad\text{and}\quad  S_1=\{\mfp\in S\mid d_1(g)\notin \mfp\}.$$ Hypothesis~{\rm\ref{26.6.1}} guarantees that $d_1(F_1')\not\subseteq \mfp_0$ for any $\mfp_0$ in $S_0$; so by the prime avoidance lemma there exists an element $\theta'\in F_1'$ with $d_1(\theta')\notin \mfp_0$ for any $\mfp_0$ in $S_0$.
In a similar manner, if $\mfp \in S_1$, then
$\mfp$ is not contained in the union of the ideals of $S_0$ and there exists an element $p_{\mfp}\in \mfp$; but $p_\mfp\notin \mfp_0$ for any $\mfp_0\in S_0$.
 Let
$$\theta=\Big(\prod_{\mfp\in S_1}p_{\mfp}\Big)\theta'.$$ Observe that $d_1(g+\theta)$ is not in any element of $S$. Indeed, if $\mfp\in S_0$, then $d_1(g)\in \mfp$ but $d_1(\theta)\notin \mfp$; but if $\mfp\in S_1$, then $d_1(g)\notin \mfp$ but $d_1(\theta)\in \mfp$. Assertion~(\ref{26.6.2}) has been established; the proof is complete.
\end{proof}

\section{Properties of the data of Lemma~\ref{2^n}.}\label{Properties of the data.}

Lemma~\ref{2^n} 
 is the most complicated part of the proof of the main Theorem. 
The proof of Lemma~\ref{2^n} is carried out in Section~\ref{Inductive step.}.
We 
start with data $(\psi_1,N)$, as described in \ref{B.1}, with $\psi_1$ not an alternating map, and we 
 modify the 
 data in order to 
 produce 
data  $(\psi_1',N')$ with the property that $\psi_1'$ is closer to being an alternating map than $\psi_1$ is.  
The most important step in the 
 modification process is
Claim~\ref{29.1.14}.
 The ultimate goal of the present section  is Corollary~\ref{X.27}, which 
 provides the proof of Claim~\ref{29.1.14}. In order to prove Corollary~\ref{X.27} we
 collect properties of the initial data $(\psi_1,N)$.

  A primitive version of Corollary~\ref{X.27} may be found in \cite[(12)]{KM80}. The present version is an improvement of  the primitive version because the present version does not require that $F_1$ has a generating set $\{e_i\}$ with the property that $d_1(e_i)$ is a regular element of $P$ for each $i$. Also, 
 the present version 
makes use of a homomorphism 
$$r:D_2F_1\to F_4$$ (see Lemma~\ref{X.16}); whereas, the corresponding object in \cite{KM80},  is  a sequence of elements $r_1,\dots, r_k$ of $F_4$. No interesting homomorphism can be constructed from these elements.

\begin{remark}\label{7.0} If the data $(\psi_1,N)$ satisfies (\ref{B.4}), then the composition
$$D_2F_1\xrightarrow{\comult}T_2F_1\xrightarrow{\psi_1}F_2$$ is in $\ker\Big(d_{2*}:\Hom(D_2F_1,F_2)\to \Hom(D_2F_1,F_1)\Big)$.
The complex $\Hom(D_2F_1,F)$ is acyclic; consequently there exists a homomorphism $\chi\in \Hom(D_2F_1,F_3)$ with 
\begin{equation}\label{B.1.5}
(d_3\circ \chi)(x_1^{(2)})=\psi_1(x_1\t x_1),
\end{equation} for $x_1\in F_1$. Furthermore, if $F_1=F_1'\p F_1''$ and $\psi_1(x_1\t x_1)=0$ for all $x_1\in F_1'$, then one can choose $\chi$ to satisfy $\chi|_{D_2F_1'}=0$ because $D_2F_1=D_2F_1'\p (F_1'\t F_1'')\p D_2F_1''$.
\end{remark}

\begin{data}
\label{data4} Adopt the data of \ref{B.1}, together with a homomorphism $\chi:
D_2F_1\to F_3$ which satisfies {\rm(\ref{B.1.5})}.\end{data}
\begin{lemma}
\label{!real-24.6} Adopt the data of {\rm\ref{data4}}. Then there exists a homomorphism 
$$\beta: D_2F_1\t F_1\to F_4$$ with
$$d_1(w_1)\cdot\beta(x_1^{(2)}\t z_1)=\begin{cases} 
\phantom{+}\psi_4\Big(\psi_1(x_1\t z_1)\cdot\psi_1(x_1\t w_1)\Big)\\
+Nd_1(z_1)\cdot \psi_3\Big(w_1\t \chi(x_1^{(2)})\Big),\end{cases}$$
for $x_1,z_1,w_1$ in $F_1$.
\end{lemma}

\begin{proof} Let $B:D_2F_1\t T_2(F_1)\to F_4$ be the homomorphism
$$B(x_1^{(2)}\t z_1\t w_1)=\begin{cases} 
\phantom{+}\psi_4\Big(\psi_1(x_1\t z_1)\cdot\psi_1(x_1\t w_1)\Big)\\
+Nd_1(z_1)\cdot \psi_3\Big(w_1\t \chi(x_1^{(2)})\Big).\end{cases}$$
According to Remark~\ref{R1} it suffices to prove that
$$B(x_1^{(2)}\t z_1\t d_2(w_2))=0,$$ 
for all $x_1,z_1\in F_1$ and $w_2\in F_2$. 
 Observe that
\begingroup\allowdisplaybreaks
\begin{align*}
&B(x_1^{(2)}\t z_1\t d_2(w_2))\\
{}={}&\begin{cases} 
\phantom{+}\psi_4\Big(\psi_1(x_1\t z_1)\cdot\psi_1(x_1\t d_2(w_2))\Big)\\
+Nd_1(z_1)\cdot \psi_3\Big(d_2(w_2)\t \chi(x_1^{(2)})\Big)\end{cases}\\
{}={}&\begin{cases} 
-\psi_4\Big(w_2\cdot\psi_1\big(x_1\t (d_2\circ\psi_1)(x_1\t z_1)\big)\Big)\\
+Nd_1(x_1)\cdot \psi_4\Big(w_2\cdot \psi_1(x_1\t z_1)\Big)\\
-Nd_1(z_1)\cdot \psi_4\Big(w_2\cdot (d_3\circ\chi)(x_1^{(2)})\Big),
\end{cases}&&\text{by (\ref{B.5}) and (\ref{A.3}),}\\
{}={}&\begin{cases} 
-Nd_1(x_1)\cdot \psi_4\Big(w_2\cdot\psi_1\big(x_1\t z_1 \big)\Big)\\
+Nd_1(z_1)\cdot\psi_4\Big(w_2\cdot\psi_1\big(x_1\t  x_1 \big)\Big)\\
+Nd_1(x_1)\cdot \psi_4\Big(w_2\cdot \psi_1(x_1\t z_1)\Big)\\
-Nd_1(z_1)\cdot \psi_4\Big(w_2\cdot \psi_1(x_1\t x_1)\Big),\end{cases}&&\text{by (\ref{B.4}) and (\ref{B.1.5})},\end{align*}\endgroup and this is zero.
\end{proof}

\begin{remark}\label{j11.1}Adopt the data of {\rm\ref{data4}}.  We often show that an element $x_4$ of $F_4$ is zero by showing that $d_1(w_1)\cdot x_4=0$ for all $w_1\in F_1$.
\begin{proof} The ideal $\im d_1$ of $P$ has positive grade; hence,  
the only element of $F_4$  annihilated by $\im d_1$ is the element zero.
(Keep in mind that $F_4$ is isomorphic to $P$.) \end{proof}\end{remark}

\begin{lemma}\label{X.16} Adopt the data of {\rm\ref{data4}} and the  homomorphism 
$\beta$ of Lemma~{\rm\ref{!real-24.6}}. Then there exists a homomorphism $r:D_2F_1\to F_4$ such that 
$$d_1(z_1)r(\theta_2)=\beta(\theta_2\t z_1)+N\psi_3(z_1\t \chi(\theta_2)),$$for $z_1\in F_1$ and $\theta_2\in D_2F_1$.
\end{lemma}
\begin{proof}
Let $\xi: D_2F_1\t F_1\to F_4$ be the map
$$\xi(\theta_2\t z)=\beta(\theta_2\t z_1)+N\psi_3(z_1\t \chi(\theta_2)).$$
According to Remark~\ref{R1} it suffices to show  
 that the element $\xi(\theta_2\t d_2(w_2))$ of $F_4$  is zero for all $\theta_2\in D_2F_1$ and $w_2\in F_2$.
In light of Remark~\ref{j11.1}, we
prove that
\begin{equation}\label{!X.3}d_1(w_1)\cdot \xi(x_1^{(2)}\t d_2(w_2))=0,\end{equation}  for all $x_1,w_1\in F_1$ and $w_2\in F_2$. Of course, 
\begin{equation}\notag\xi(x_1^{(2)}\t d_2(w_2))=\beta(x_1^{(2)}\t d_2(w_2))+N\psi_3(d_2(w_2)\t \chi(x_1^{(2)}))\end{equation} and one applies (\ref{A.3}) and (\ref{B.1.5}) to see that
\begin{align}\psi_3(d_2(w_2)\t \chi(x_1^{(2)})){}={}&{}-\psi_4(w_2\cdot (d_3 \circ\chi)(x_1^{(2)}))\label{X.16.2}\\
{}={}&{}-\psi_4(w_2\cdot \psi_1(x_1\t x_1)).\notag\end{align}
On the other hand, 
\begin{align*}&d_1(w_1)\beta(x_1^{(2)}\t d_2(w_2))\\{}={}&\psi_4\Big(\psi_1\big(x_1\t d_2(w_2)\big)\cdot \psi_1(x_1\t w_1)\Big),&&\text{by Lemma~\ref{!real-24.6}},\\
{}={}&\begin{cases} 
-\psi_4\Big(\psi_1\big(x_1\t (d_2\circ\psi_1)(x_1\t w_1)\big)\cdot w_2\Big)\\
+N d_1(x_1)\cdot \psi_4\Big( \psi(x_1\t w_1)\cdot  w_2\Big),\end{cases}&&\text{by (\ref{B.5}),}\\
{}={}&\begin{cases} 
-N d_1(x_1)\cdot\psi_4\Big(\psi_1(x_1\t  w_1))\cdot w_2\Big)\\
+N d_1(w_1)\cdot\psi_4\Big(\psi_1(x_1\t x_1)\cdot w_2\Big)\\
+N d_1(x_1)\cdot\psi_4(\psi(x_1\t w_1)\cdot w_2),\end{cases}&&\text{by (\ref{B.4}),}\\
{}={}&{}-Nd_1(w_1)\cdot \psi_3(d_2(w_2)\t \chi(x_1^{(2)})),&&\text{by (\ref{X.16.2})}.
\end{align*} 
Thus, $$d_1(w_1)\Big[\beta(x_1^{(2)}\t d_2(w_2))+N \psi_3(d_2(w_2)\t \chi(x_1^{(2)}))\Big]=0;$$
(\ref{!X.3}) is established; and the proof is complete.
\end{proof}

\begin{corollary}\label{X.3}Adopt the data of {\rm\ref{data4}} and the  homomorphism
 $r$ of Lemma~{\rm\ref{X.16}}. If $x_1,z_1,w_1$ are in $F_1$, then
$$d_1(w_1)\cdot d_1(z_1)\cdot r(x_1^{(2)})=\begin{cases}
\phantom{+}\psi_4\Big(\psi_1(x_1\t z_1)\cdot \psi_1(x_1\t w_1)\Big)\\
+Nd_1(z_1)\cdot \psi_3\Big(w_1\t \chi(x_1^{(2)})\Big)\\
+Nd_1(w_1)\cdot \psi_3\Big(z_1\t \chi(x_1^{(2)})\Big)
\end{cases}$$
\end{corollary}

\begin{proof}Recall from Lemma~\ref{!real-24.6} that 
$$d_1(w_1)\cdot\beta(x_1^{(2)}\t z_1)=\begin{cases} 
\phantom{+}\psi_4\Big(\psi_1(x_1\t z_1)\cdot\psi_1(x_1\t w_1)\Big)\\
+Nd_1(z_1)\cdot \psi_3\Big(w_1\t \chi(x_1^{(2)})\Big).\end{cases}$$ and from Lemma~\ref{X.16} that 
$$d_1(z_1)r(x_1^{(2)})=\beta(x_1^{(2)}\t z_1)+N\psi_3(z_1\t \chi(x_1^{(2)})).$$
\vskip-24pt\end{proof}

\begin{lemma}\label{X.19} Adopt the data of {\rm\ref{data4}} and the  homomorphism 
$\beta$  of Lemmas~{\rm\ref{!real-24.6}}.  If $x_1$ is in $F_1$, then 
$$\beta(x_1^{(2)}\t x_1)=N\psi_3\big(x_1\t \chi(x_1^{(2)})\big).$$\end{lemma}

\begin{proof} In light of Remark~\ref{j11.1}, it suffices to prove that
$$d_1(w_1)\cdot \beta(x_1^{(2)}\t x_1)=Nd_1(w_1)\cdot \psi_3(x_1\t \chi(x_1^{(2)})),$$ for all $x_1, w_1$ in $F_1$.
Observe that
\begingroup\allowdisplaybreaks
\begin{align*}
&d_1(w_1)\cdot\beta(x_1^{(2)}\t x_1)\\
{}={}&\begin{cases} 
\phantom{+}\psi_4\Big(\psi_1(x_1\t x_1)\cdot\psi_1(x_1\t w_1)\Big)\\
+Nd_1(x_1)\cdot \psi_3\Big(w_1\t \chi(x_1^{(2)})\Big),\end{cases}&&\text{by Lemma~\ref{!real-24.6}},\\
{}={}&\begin{cases} 
\phantom{+}\psi_4\Big(d_3\chi(x_1^{(2)})\cdot\psi_1(x_1\t w_1)\Big)\\
+Nd_1(x_1)\cdot \psi_3\Big(w_1\t \chi(x_1^{(2)})\Big),\end{cases}&&\text{by (\ref{B.1.5}),}\\
{}={}&\begin{cases} 
-\psi_3\Big((d_2\circ \psi_1)(x_1\t w_1)\t \chi(x_1^{(2)})\Big)\\
+Nd_1(x_1)\cdot \psi_3\Big(w_1\t \chi(x_1^{(2)})\Big),\end{cases}
&&\text{by (\ref{A.3})}\\
{}={}&\begin{cases} 
-Nd_1(x_1)\cdot\psi_3\Big(w_1\t \chi(x_1^{(2)})\Big)\\
+Nd_1(w_1)\cdot\psi_3\Big(x_1\t \chi(x_1^{(2)})\Big)\\
+Nd_1(x_1)\cdot \psi_3\Big(w_1\t \chi(x_1^{(2)})\Big),\end{cases}&&\text{by (\ref{B.4}),}\\
{}={}& 
Nd_1(w_1)\cdot \psi_3\Big(x_1\t \chi(x_1^{(2)})\Big).
\end{align*}\endgroup
\vskip-24pt\end{proof}

\begin{corollary}\label{X.20} Adopt the data of {\rm\ref{data4}} and the  homomorphism $r$ of Lemmas~{\rm\ref{X.16}}. 
 If $x_1\in F_1$, then 
$$d_1(x_1)r(x_1^{(2)})=2N\psi_3\big(x_1\t \chi(x_1^{(2)})\big).$$\
\end{corollary}
\begin{proof}Observe that 
\begin{align*}d_1(x_1)r(x_1^{(2)}){}={}&
\beta(x_1^{(2)}\t x_1)+N\psi_3(x_1\t \chi(x_1^{(2)})),&&\text
{by Lemma~\ref{X.16},}\\{}={}&2N\psi_3(x_1\t \chi(x_1^{(2)})),&&\text{by Lemma~\ref{X.19}.}\end{align*}
\end{proof}

\begin{lemma}\label{X.75} Adopt the data of {\rm\ref{data4}} and the  homomorphism 
$\beta$  of Lemma~{\rm\ref{!real-24.6}}. If $x_1$ and $y_1$ are in $F_1$, then 
$$\beta(y_1x_1\t x_1)+\beta(x_1^{(2)}\t y_1)=N\psi_3\Big(x_1\t\chi(x_1y_1)\Big)+N\psi_3\Big(y_1\t\chi(x_1^{(2)})\Big).$$
\end{lemma}

\begin{proof} We apply Remark~\ref{j11.1} and show that proposed equality holds after both sides have been multiplied by $d_1(w_1)$, for each $w_1\in F_1$.
Apply Lemma~\ref{!real-24.6} twice to write
$$d_1(w_1)\cdot\Big[\beta(y_1x_1\t x_1)+\beta({x_1}^{(2)}\t y_1)\Big]=\sum_{i=1}^5S_i,$$ with
\begingroup\allowdisplaybreaks\begin{align*}
S_1={}&\psi_4\Big(\psi_1(x_1\t x_1)\cdot\psi_1(y_1\t w_1)\Big),\\
S_2={}&\psi_4\Big(\psi_1(y_1\t x_1)\cdot\psi_1(x_1\t w_1)\Big),\\
S_3={}&Nd_1(x_1)\cdot \psi_3\Big(w_1\t \chi(x_1y_1)\Big),\\
S_4={}&\psi_4\Big(\psi_1(x_1\t y_1)\cdot\psi_1(x_1\t w_1)\Big),\text{ and}\\
S_5={}&Nd_1(y_1)\cdot \psi_3\Big(w_1\t \chi(x_1^{(2)})\Big).\\
\end{align*}\endgroup
Observe that
\begingroup\allowdisplaybreaks\begin{align*}
S_2+S_4={}&\psi_4\Big((d_3\circ\chi)(x_1y_1)\cdot\psi_1(x_1\t w_1)\Big),&&\text{by (\ref{B.1.5}),}\\
{}={}&{}-\psi_3\Big(d_2\psi_1(x_1\t w_1)\t\chi(x_1y_1)\Big),&&\text{by (\ref{A.3}),}\\
{}={}&{}-S_3+Nd_1(w_1)\cdot  \psi_3\Big(x_1
\t\chi(x_1y_1)\Big),&&\text{by (\ref{B.4}).}
\end{align*}\endgroup
Similarly,
\begingroup\allowdisplaybreaks\begin{align*}
S_1={}&\psi_4\Big(d_3\chi(x_1^{(2)})\cdot\psi_1(y_1\t w_1)\Big),&&\text{by (\ref{B.1.5}),}\\
{}={}&{}-\psi_3\Big(d_2\psi_1(y_1\t w_1)\t\chi(x_1^{(2)})\Big),&&\text{by (\ref{A.3}),}\\
{}={}&{}-S_5+Nd_1(w_1)\cdot\psi_3\Big(y_1\t\chi(x_1^{(2)})\Big),&&\text{by (\ref{B.4}).}
\end{align*}\endgroup
Thus, $$\sum_{i=1}^5 S_i=d_1(w_1)\cdot 
\Big[N\psi_3\Big(x_1\t\chi(x_1y_1)\Big)+N\psi_3\Big(y_1\t\chi(x_1^{(2)})\Big)\Big],$$ and the proof is complete. \end{proof}

\begin{corollary}\label{X.27} Adopt the data of {\rm\ref{data4}} and the  homomorphism $r$
of Lemma~{\rm\ref{X.16}}. 
If $x_1$ and $y_1$ are in $F_1$, then 
$$d_1(x_1)\cdot r(x_1y_1)+d_1(y_1)r(x_1^{(2)})=2N\psi_3(x_1\t \chi(x_1y_1))+2N\psi_3(y_1\t \chi(x_1^{(2)})).$$
\end{corollary}

\begin{proof} Observe that
\begin{align*}
&d_1(x_1)\cdot r(x_1y_1)+d_1(y_1)r(x_1^{(2)})\\
{}={}&\begin{cases}
\phantom{+}\beta(x_1y_1\t x_1)+N\psi_3(x_1\t \chi(x_1y_1))\\
+\beta(x_1^{(2)}\t y_1)+N\psi_3(y_1\t \chi(x_1^{(2)})),\end{cases}&&\text{by  Lemma~\ref{X.16},}\\
{}={}&\begin{cases}\phantom{+}2N\psi_3(x_1\t \chi(x_1y_1))\\+2N\psi_3(y_1\t \chi(x_1^{(2)})),\end{cases}&&\text{by  Lemma~\ref{X.75}}.\end{align*}
\vskip-24pt\end{proof}

\section{The induction step in the proof of Lemma~\ref{2^n}.}\label{Inductive step.}

The goal of this section is \ref{prove 2^n}, where  we prove Lemma~\ref{2^n}. 
The main step in this proof is carried out in Proposition~\ref{prop29.5}, where 
we 
start with data $(\psi_1,N,\chi)$, as described in \ref{data4}, with $\psi_1$ not an alternating map, and we 
 modify the 
 data in order to 
 produce 
data  $(\psi_1',N',\chi')$ with the property that $\psi_1'$ is closer to being an alternating map than $\psi_1$ is. 

The proof of Lemma~\ref{2^n} is inspired by, but significantly different than, the proof in \cite{KM80}.

\begin{data}\label{29.1} Adopt the data of {\rm\ref{data4}}. 
Assume further that 
there is a decomposition \begin{equation}F_1=G_1\p H_1\p I_1\label{29.1.b}\end{equation}  of $F_1$ 
as a direct sum of free $P$ modules, where $H_1$ has rank $1$,
\begin{enumerate}[\rm(a)] 
\item \label{29.5.5+}$\chi|_{D_2G_1}=0$, and 
\item\label{29.1.d} if $G_1$ is not zero, then there is an element $\zeta_1\in G_1$ with $d_1(\zeta_1)$ is regular on $P$.
\end{enumerate}
\end{data}

\begin{remark}\label{Rem29.3}If the hypotheses of \ref{29.1} are in effect, then
it follows from (\ref{B.1.5}) and \ref{29.1}.(\ref{29.5.5+}) that 
$$\psi_1(x_1\t x_1)=0\quad\text{and}\quad \psi_1(x_1\t y_1)+\psi_1(y_1\t x_1)=0,$$for all $x_1$ and $y_1$ in $G_1$.
\end{remark}

\begin{lemma}\label{L29.2}
Adopt the  Data of {\rm\ref{29.1}}. The following statements hold\,{\rm:}
\begin{enumerate}[\rm(a)]
\item\label{L29.2.a} $r|_{D_2G_1}=0$ and 
\item\label{L29.2.b} $
\psi_4(\psi_1(x_1\t w_1)\cdot \psi_1(x_1\t z_1))=0$, for all $x_1\in G_1$ and $z_1,w_1\in F_1$.
\end{enumerate}
\end{lemma}

\begin{proof} 
If $G_1=0$, then (\ref{L29.2.a}) and  (\ref{L29.2.b}) automatically hold. Henceforth, we assume $G_1$ is not zero. According to Data~\ref{29.1}.(\ref{29.1.d}), there is an element $\zeta_1$ in $G_1$ with $d_1(\zeta_1)$ regular on $P$. We 
first prove some further properties of $\zeta_1$. 
Apply 
Corollary~\ref{X.20}, with $x_1$ replaced by $\zeta_1$, to see that
$$d_1(\zeta_1)r(\zeta_1^{(2)})=2N\psi_3\big(\zeta_1\t \chi(\zeta_1^{(2)})\big).$$The element $\chi(\zeta_1^{(2)})$ is zero by Data~\ref{29.1}.(\ref{29.5.5+}) and the element $d_1(\zeta_1)$ is regular on $F_4$ (which is isomorphic to $P$); therefore, we conclude that $r(\zeta_1^{(2)})=0$.
 Now apply Corollary~\ref{X.3}, with $x_1$ replaced by $\zeta_1$, to see that
$$d_1(w_1)\cdot d_1(z_1)\cdot r(\zeta_1^{(2)})=\begin{cases}
\phantom{+}\psi_4\Big(\psi_1(\zeta_1\t z_1)\cdot \psi_1(\zeta_1\t w_1)\Big)\\
+Nd_1(z_1)\cdot \psi_3\Big(w_1\t \chi(\zeta_1^{(2)})\Big)\\
+Nd_1(w_1)\cdot \psi_3\Big(z_1\t \chi(\zeta_1^{(2)})\Big).
\end{cases}$$
The elements $r(\zeta_1^{(2)})$ of $F_4$ and $\chi(\zeta_1^{(2)})$ of $F_3$ are zero. It follows that 
$$\psi_4\Big(\psi_1(\zeta_1\t z_1)\cdot \psi_1(\zeta_1\t w_1)\Big)=0,$$
for all $z_1$ and $w_1$ of $F_1$.
Apply Remark~\ref{Rem29.3} to conclude that
\begin{equation}\label{29.2.1}\psi_4\Big(\psi_1(z_1\t \zeta_1)\cdot \psi_1(w_1\t \zeta_1)\Big)=0,\end{equation}
for $z_1$ and $w_1$ in $G_1$.

Now we attack assertion (\ref{L29.2.a}).
Apply Corollary~\ref{X.3} again. This time, 
take $z_1$ and $w_1$ both to be $\zeta_1$ and let $x_1$  be an arbitrary element of $G_1$. Obtain
$$d_1(\zeta_1)\cdot d_1(\zeta_1)\cdot r(x_1^{(2)})=\begin{cases}
\phantom{+}\psi_4\Big(\psi_1(x_1\t \zeta_1)\cdot \psi_1(x_1\t \zeta_1)\Big)
\\
+Nd_1(\zeta_1)\cdot \psi_3\Big(\zeta_1\t \chi(x_1^{(2)})\Big)
\\
+Nd_1(\zeta_1)\cdot \psi_3\Big(\zeta_1\t \chi(x_1^{(2)})\Big)
\end{cases}$$
The top summand on the right is zero by (\ref{29.2.1}). The other two summands on the right are zero because $\chi(D_2(G_1))=0$ by Data~\ref{29.1}.(\ref{29.5.5+}). It follows that the element
$$d_1(\zeta_1)\cdot d_1(\zeta_1)\cdot r(x_1^{(2)})$$ of $F_4$ is zero. However, $d_1(\zeta_1)$ is regular on $F_4$; hence $r(x_1^{(2)})=0$ for all $x_1$ in $G_1$ and the restriction of $r$ to $D_2G_1$ is identically zero. This is assertion (\ref{L29.2.a}).

 For  assertion (\ref{L29.2.b}), apply Corollary~\ref{X.3} again.
If $x_1\in G_1$ and $z_1,w_1$ are in $F_1$, then
$$d_1(w_1)\cdot d_1(z_1)\cdot r(x_1^{(2)})=\begin{cases}
\phantom{+}\psi_4\Big(\psi_1(x_1\t z_1)\cdot \psi_1(x_1\t w_1)\Big)\\
+Nd_1(z_1)\cdot \psi_3\Big(w_1\t \chi(x_1^{(2)})\Big)\\
+Nd_1(w_1)\cdot \psi_3\Big(z_1\t \chi(x_1^{(2)})\Big).
\end{cases}$$ 
 The elements $r(x_1^{(2)})$ of $F_4$ and $\chi(x_1^{(2)})$ of $F_3$ are zero by (\ref{L29.2.a}) and Data~\ref{29.1}.(\ref{29.5.5+}). It follows that 
$$\psi_4\Big(\psi_1(x_1\t z_1)\cdot \psi_1(x_1\t w_1)\Big)=0.$$
This is assertion (\ref{L29.2.b}).
\end{proof}

\begin{proposition}
\label{prop29.5}Adopt Data~{\rm\ref{29.1}}. Let $N'$ be the integer $N'=2N^2$.
Then there exists a
homomorphism $\psi_1':T_2F_1\to F_2$ 
 such that 
\begin{equation}\label{29.4'}d_2\circ \psi_1'(x_1\t y_1)= N'\Big(d_1(x_1)\cdot y_1-d_1(y_1)\cdot x_1\Big),\end{equation}
\begin{equation}\label{29.5'}\psi_4\Big(\psi_1'(x_1\t d_2(x_2))\cdot x_2\Big)=N'(d_1(x_1))\cdot \psi_4( x_2^{(2)}), \quad\text{and}\end{equation}
\begin{equation}\label{29.5.5'} \psi_1'|_{\im D_2(G_1\p H_1)}=0. \end{equation}
\end{proposition}

\begin{proof}
 Let $\proj_{G_1\p H_1}:F_1\to G_1\p H_1$ be the projection map induced by the direct sum decomposition of $F_1$ which is given in (\ref{29.1.b}), $h_1$ be a basis element for $H_1$, $h_1^*:F_1\to P$ be the homomorphism which sends $G_1$ and $I_1$ to zero and $h_1$ to $1$, and $\int:F_1\to D_2F_1$ be the homomorphism
$${\ts\int} (x_1)=\begin{cases} x_1h_1&\text{if $x_1\in G_1\p I_1$}\\h_1^{(2)}&\text{if $x_1=h_1$}
.\end{cases}$$
Let 
$\NNY:T_2F_1\to F_3$ be the homomorphism defined by 
\begin{align}\label{29.5.4}&\psi_3\Big(x_1\t \NNY(y_1\t z_1)\Big)\\\notag{}={}& h_1^*(y_1)\begin{cases} 
+2N\cdot \psi_3\Big(x_1\t \chi\big((\int \circ\proj_{G_1\p H_1})(z_1)\big)\Big)\\
-2N\cdot \psi_3\Big(z_1\t \chi \big((\int \circ\proj_{G_1\p H_1})(x_1)\big)\Big)\\ 
+2N\cdot \psi_3\Big(\proj_{G_1\p H_1}(z_1)\t \chi\big((\int \circ\proj_{G_1\p H_1})(x_1)\big)\Big)\\
-d_1(\proj_{G_1\p H_1}(z_1))\cdot r\big((\int \circ\proj_{G_1\p H_1})(x_1)\big)\\
+d_1(z_1)\cdot r\big((\int \circ\proj_{G_1\p H_1})(x_1)\big),
\end{cases}\end{align}
and $\psi_1':T_2(F_1)\to F_2$ be the homomorphism  
$$\psi_1'=2N\psi_1-d_3\circ \NNY.$$

We first verify (\ref{29.4'}).  Observe that
\begingroup\allowdisplaybreaks
\begin{align*}d_2\circ \psi_1'(x_1\t y_1){}={}&(d_2\circ (2N\psi_1-d_3\circ \NNY))(x_1\t y_1)\\
{}={}&2N(d_2\circ \psi_1)(x_1\t y_1),&&\text{because $d_2\circ d_3=0$,}\\
{}={}&2N^2\Big(d_1(x_1)\cdot y_1-d_1(y_1)\cdot x_1\Big),&&\text{by (\ref{B.4}),}\\
{}={}&N'\Big(d_1(x_1)\cdot y_1-d_1(y_1)\cdot x_1\Big),\end{align*}\endgroup as desired.

Now we verify (\ref{29.5'}). Observe that
\begingroup\allowdisplaybreaks
\begin{align*}
&\psi_4\Big(\psi_1'(x_1\t d_2(x_2))\cdot x_2\Big)\\{}={}
&\psi_4\Big(
(2N\psi_1-d_3\circ \NNY)(x_1\t d_2(x_2))\cdot x_2\Big)\\{}={}
&\begin{cases}\phantom{+}2N\psi_4\Big(\psi_1(x_1\t d_2(x_2))\cdot x_2\Big)\\
-\psi_4\Big((d_3\circ \NNY)(x_1\t d_2(x_2))\cdot x_2\Big)\end{cases}\\{}={}
&\begin{cases}\phantom{+}2N\cdot 
N\cdot d_1(x_1)\cdot \psi_4(x_2^{(2)})\\
+\psi_3\Big(d_2(x_2)\t \NNY(x_1\t d_2(x_2))\Big),\end{cases}&&\text{by (\ref{B.5}) and (\ref{A.3}),}
\\{}={}
&\begin{cases}\phantom{+}N'\cdot d_1(x_1)\cdot \psi_4(x_2^{(2)})\\
+\psi_3\Big(d_2(x_2)\cdot \NNY(x_1\t d_2(x_2))\Big).\end{cases}
\end{align*}\endgroup
We complete the verification  of (\ref{29.5'}) by showing that
$$\psi_3\Big(d_2(x_2)\t \NNY(x_1\t d_2(x_2))\Big)=0.$$
Use (\ref{29.5.4}) to see that 
$$
  \psi_3\Big(d_2(x_2)\t \NNY(x_1\t d_2(x_2))\Big)$$$$= h_1^*(x_1)\begin{cases} 
+2N\cdot \psi_3\Big(d_2(x_2)\t \chi\big((\int \circ\proj_{G_1\p H_1})(d_2(x_2))\big)\Big)\\
-2N\cdot \psi_3\Big(d_2(x_2)\t \chi \big((\int \circ\proj_{G_1\p H_1})(d_2(x_2))\big)\Big)\\ 
+2N\cdot \psi_3\Big(\proj_{G_1\p H_1}(d_2(x_2))\t \chi\big((\int \circ\proj_{G_1\p H_1})(d_2(x_2))\big)\Big)\\
-d_1(\proj_{G_1\p H_1}(d_2(x_2)))\cdot r\big((\int \circ\proj_{G_1\p H_1})(d_2(x_2))\big)\\
+d_1(d_2(x_2))\cdot r\big((\int \circ\proj_{G_1\p H_1})(d_2(x_2))\big)
\end{cases}$$The top two summands add to zero. The bottom summand is zero.  Apply Claim \ref{29.1.14} to see that summands three and four add to zero.

\begin{claim-no-advance}\label{29.1.14}
If $y_1$ is in $G_1\p H_1$, then 
$$\ts 2N\cdot \psi_3(y_1\t \chi(\int y_1))
-d_1(y_1)\cdot r(\int y_1)$$ is equal to zero.\end{claim-no-advance} 

\ms\noindent{\it Proof of Claim~{\rm\ref{29.1.14}}.} It suffices to write $y_1=g_1+p h_1$, with $g_1\in G_1$ and $p\in P$ and to prove that
\begingroup\allowdisplaybreaks\begin{align}
0={}&\ts 2N\cdot \psi_3(g_1\t \chi(\int g_1))
-d_1(g_1)\cdot r(\int g_1),\label{first formula}\\
0={}&\ts \begin{cases}\phantom{+}2N\cdot \psi_3(g_1\t \chi(\int h_1))
-d_1(g_1)\cdot r(\int h_1)\\+2N\cdot \psi_3(h_1\t \chi(\int g_1))
-d_1(h_1)\cdot r(\int g_1),&\text{and}\end{cases}\label{second formula}\\
0={}&\ts 2N\cdot \psi_3(h_1\t \chi(\int h_1))
-d_1(h_1)\cdot r(\int h_1).\label{third formula}
\end{align}\endgroup

Formulas (\ref{third formula}) and (\ref{second formula}) are  immediate applications of Corollaries \ref{X.20} and 
\ref{X.27}, respectively. For formula (\ref{first formula}), apply Corollary~\ref{X.27} with
$x_1$ replaced by $g_1$ and $y_1$ replaced by $h_1$:
$$d_1(g_1)\cdot r(g_1h_1)+d_1(h_1)r(g_1^{(2)})=2N\psi_3(g_1\t \chi(g_1h_1))+2N\psi_3(h_1\t \chi(g_1^{(2)})).$$
Recall from Data~\ref{29.1}.(\ref{29.5.5+}) and Lemma~\ref{L29.2} that 
$\chi(g_1^{(2)})$  and  
$r(g_1^{(2)})$
are both zero. Thus $$d_1(g_1)\cdot r(g_1h_1)=2N\psi_3(g_1\t \chi(g_1h_1)),
$$ which is (\ref{first formula}). This completes the proof of Claim~\ref{29.1.14}. \hfill \qed 

\bs To complete the proof of Proposition~\ref{prop29.5}, we verify (\ref{29.5.5'}). Indeed, we show that $\psi_1'$ vanishes on $g_1\t g_1$, $h_1\t h_1$ and $g_1\t h_1+h_1\t g_1$, with $g_1$ in $G_1$.

\begin{claim-no-advance}\label{29.5.9} If $x_1$ and $y_1$ are in $F_1$ with $h_1^*(x_1)=0$, then $(d_3\circ \NNY)(x_1\t y_1)=0$. 
\end{claim-no-advance}

\ms\noindent{\it Proof of Claim~{\rm\ref{29.5.9}}.} It suffices to show that
\begin{equation}\label{29.5.10}\psi_4\Big(x_2\cdot (d_3\circ \NNY)(x_1\t y_1)\Big)=0.\end{equation} According to (\ref{A.3}) and (\ref{29.5.4}), the left side of (\ref{29.5.10}) is
$$-\psi_3\Big(d_2(x_2)\t \NNY(x_1\t y_1)\Big),$$ which is equal to $h_1^*(x_1)$ times an element of $F_4$.
This completes the proof of Claim~\ref{29.5.9}. \hfill \qed 

\bs The homomorphism $\psi'=2N\psi_1-d_3\circ \NNY$ vanishes on $g_1\t g_1$ because of Remark~\ref{Rem29.3} and Claim~\ref{29.5.9}.

\ms We show that $\psi_1'(h_1\t h_1)=0$ by showing that $\psi_4\Big(x_2\cdot \psi_1'(h_1\t h_1)\Big)=0$ for all $x_2\in F_2$. Recall that
\begin{align*}2N\psi_4\Big(x_2\cdot \psi_1(h_1\t h_1)\Big)={}&2N\psi_4\Big(x_2\cdot (d_3\circ\chi)(h_1^{(2)})\Big),&&\text{by (\ref{B.1.5}),}\\
{}={}&-2N\psi_3\Big(d_2(x_2)\t \chi(h_1^{(2)})\Big),&&\text{by (\ref{A.3}),}
\end{align*} and 
$$\psi_4\Big(x_2\cdot (d_3\circ \NNY)(h_1\t h_1)\Big)
=
-\psi_3\Big(d_2(x_2)\t \NNY(h_1\t h_1)\Big),$$by (\ref{A.3}).
It suffices to show that
\begin{equation}2N\psi_3\Big(d_2(x_2)\t\chi(h_1^{(2)})\Big)=\psi_3\Big(d_2(x_2)\t \NNY(h_1\t h_1)\Big).\label{29.5.11}\end{equation}
The definition of $\NNY$, which is given in (\ref{29.5.4}), yields that
\begingroup\allowdisplaybreaks \begin{align*} &\psi_3\Big(d_2(x_2)\t \NNY(h_1\t h_1)\Big)\\
{}={}&h_1^*(h_1)\begin{cases} 
+2N\cdot \psi_3\Big(d_2(x_2)\t \chi\big((\int \circ\proj_{G_1\p H_1})(h_1)\big)\Big)\\
-2N\cdot \psi_3\Big(h_1\t \chi \big((\int \circ\proj_{G_1\p H_1})(d_2(x_2))\big)\Big)\\ 
+2N\cdot \psi_3\Big(\proj_{G_1\p H_1}(h_1)\t \chi\big((\int \circ\proj_{G_1\p H_1})(d_2(x_2))\big)\Big)\\
-d_1\big(\proj_{G_1\p H_1}(h_1)\big)\cdot r\big((\int \circ\proj_{G_1\p H_1})(d_2(x_2))\big)\\
+d_1(h_1)\cdot r\big((\int \circ\proj_{G_1\p H_1})(d_2(x_2))\big).
\end{cases}\end{align*}\endgroup
The second and third summands add to become zero; as do the fourth and fifth summands. Thus, (\ref{29.5.11}) is established and $\psi_1'(h_1\t h_1)=0$.

\ms Finally, we show that $\psi_1'(g_1\t h_1+h_1\t g_1)=0$.
Recall that \begin{align*}
      2N\psi_4\Big(x_2\cdot \psi_1(g_1\t h_1+h_1\t g_1)\Big)
{}={}&2N\psi_4\Big(x_2\cdot (d_3\circ\chi)(h_1g_1)\Big),&&\text{by (\ref{B.1.5})},\\
{}={}&-2N\psi_3\Big(d_2(x_2)\t \chi(h_1g_1)\Big),&&\text{by (\ref{A.3})},\end{align*}
and \begin{align*}
&\psi_4\Big(x_2\cdot (d_3\circ \NNY)(g_1\t h_1+h_1\t g_1)\Big)\\
{}={}&
\psi_4\Big(x_2\cdot (d_3\circ \NNY)(h_1\t g_1)\Big),&&\text{by Claim \ref{29.5.9}},\\
{}={}&
-\psi_3\Big(d_2(x_2)\t \NNY(h_1\t g_1)\Big),&&\text{by (\ref{A.3})}.\end{align*}
It suffices to show that \begin{equation}\label{29.5.12}2N\psi_3\Big(d_2(x_2)\t \chi(h_1g_1)\Big)=\psi_3\Big(d_2(x_2)\t \NNY(
h_1\t g_1)\Big).\end{equation} 
The definition of $\NNY$, which is given in (\ref{29.5.4}), yields that
$$
\psi_3\Big(d_2(x_2)\t \NNY(h_1\t g_1)\Big)$$$$= h_1^*(h_1)\begin{cases} 
+2N\cdot \psi_3\Big(d_2(x_2)\t \chi\big((\int \circ\proj_{G_1\p H_1})(g_1)\big)\Big)\\
-2N\cdot \psi_3\Big(g_1\t \chi \big((\int \circ\proj_{G_1\p H_1})(d_2(x_2))\big)\Big)\\ 
+2N\cdot \psi_3\Big(\proj_{G_1\p H_1}(g_1)\t \chi\big((\int \circ\proj_{G_1\p H_1})(d_2(x_2))\big)\Big)\\
-d_1\big(\proj_{G_1\p H_1}(g_1)\big)\cdot r\big((\int \circ\proj_{G_1\p H_1})(d_2(x_2))\big)\\
+d_1(g_1)\cdot r\big((\int \circ\proj_{G_1\p H_1})(d_2(x_2))\big)
\end{cases}$$
The second and third summands add to become zero; as do the fourth and fifth summands. Thus, (\ref{29.5.12}) is established and $\psi_1'(g_1\t h_1+h_1\t g_1)=0$. The proof of Proposition~\ref{prop29.5} is complete.
 \end{proof}

\begin{chunk}
\label{prove 2^n}{\em The proof of Lemma~{\rm{\ref{2^n}}}.} Start with Data~\ref{A.1}.
Apply Observation~\ref{Y22.5} to obtain a map $\psi_1^{\la 0\ra}:T_2F_1\to F_2$ which is $2$-compatible with the data of \ref{A.1}. At this point, one has the data of \ref{B.1}. According to Remark~\ref{7.0}, there exists a homomorphism $\chi:D_2F_1\to F_3$ which satisfies (\ref{B.1.5}). At this point, one has the data of \ref{data4}. 

Apply Lemma~\ref{2.2} to obtain an element $h_1\in F_1$ and a complementary free summand $I_1$ such that $F_1=Ph_1\p I_1$ and $d_1(h_1)$ is a regular element of $P$. Take $H_1=Ph_1$ and $G_1$ to be zero. Observe that conditions (\ref{29.5.5+}) and (\ref{29.1.d}) of Data~\ref{29.1} are automatically satisfied. Indeed, all of Data~\ref{29.1} has been accumulated.

Apply Proposition~\ref{prop29.5} to produce a homomorphism $\psi_1':T_2F_1\to F_2$ and an integer $N'=2^{n'}$, for some positive integer $n'$, so that $\psi_1'$ is $N'$-compatible with the data of \ref{A.1} and $\psi_1'|_{\im D_2(G_1\p H_1)}=0$. Apply Remark~\ref{7.0} again to produce $\chi':D_2F_1\to F_3$, with $(d_3\circ \chi')(x_1^{(2)})=\psi_1'(x_1\t x_1)$ and $\chi'|_{D_2(G_1\p H_1)}=0$.

Pick a new decomposition $G_1'\p H_1'\p I_1'$ of $F_1$ with $G_1'$ equal to $G_1\p H_1$ and $\rank H_1'=1$. Apply Proposition~\ref{prop29.5} again to produce a homomorphism $$\psi_1'':T_2F_1\to F_2$$ and an integer $N''=2^{n''}$
so that $\psi_1''$ is $N''$-compatible with the data of \ref{A.1} and 
$\chi'|_{D_2(G_1'\p H_1')}=0$. Of course, $G_1\p H_1$ and $G_1'\p H_1'$ are both free summands of $F_1$; but
$$\rank(G_1\p H_1)+1=\rank(G_1'\p H_1').$$ One iterates the procedure 
to produce an alternating homomorphism $$\Psi_{1,2}:T_2F_1\to F_2$$ which is $2^m$-compatible with the data of \ref{A.1}, for some non-negative integer $m$.
 The proof of Lemma~\ref{2^n} is complete.\hfill \qed
\end{chunk}

\end{document}